\documentclass[12pt]{amsart}
\usepackage{amssymb,amsthm}
\usepackage{hyperref} 
\usepackage{dsfont}	
\usepackage[nospace,noadjust]{cite}		
\usepackage{color}

\usepackage[all]{xy}
\newtheorem{theorem}{Theorem}[section]
\newtheorem{proposition}[theorem]{Proposition}
\newtheorem{coro}[theorem]{Corollary}
\newtheorem{lemma}[theorem]{Lemma}
\newtheorem{rem}[theorem]{Remark}
\newtheorem{definition}[theorem]{Definition}

\theoremstyle{plain}

\renewcommand{\epsilon}{\varepsilon}
\newcommand{\eps}{\epsilon}

\newcommand{\diam}{\textrm{diam}}

\newcommand\N{\mathbb{N}}
\newcommand\R{\mathbb{R}}

\DeclareMathOperator{\osc}{Osc}

\newcommand{\abs}[1]{\left|#1\right|}

\DeclareMathOperator{\dist}{dist}
\DeclareMathOperator{\supp}{supp}
\DeclareMathOperator{\loc}{loc}
\DeclareMathOperator*{\esssup}{ess\,sup}
\newcommand{\Eins}{\ensuremath{\mathds{1}}}

\def\Xint#1{\mathchoice
   {\XXint\displaystyle\textstyle{#1}}%
   {\XXint\textstyle\scriptstyle{#1}}%
   {\XXint\scriptstyle\scriptscriptstyle{#1}}%
   {\XXint\scriptscriptstyle\scriptscriptstyle{#1}}%
   \!\int}
\def\XXint#1#2#3{{\setbox0=\hbox{$#1{#2#3}{\int}$}
     \vcenter{\hbox{$#2#3$}}\kern-.5\wd0}}

\def\aver#1{\Xint-_{#1}}

\makeatletter
\let\original@addcontentsline\addcontentsline
\newcommand{\dummy@addcontentsline}[3]{}
\newcommand{\DeactivateToc}{\let\addcontentsline\dummy@addcontentsline}
\newcommand{\ActivateToc}{\let\addcontentsline\original@addcontentsline}
\makeatother

\newcommand\restr[2]{{
  \left.\kern-\nulldelimiterspace 
  #1 
  \vphantom{\big|} 
  \right|_{#2} 
  }}

\numberwithin{theorem}{section}
\numberwithin{equation}{section}

\title[Gaussian bounds through elliptic Moser]{Gaussian heat kernel bounds through elliptic Moser iteration}

\author[Fr\'ed\'eric Bernicot, Thierry Coulhon, Dorothee Frey]{Fr\'ed\'eric Bernicot, Thierry Coulhon, Dorothee Frey}

\address{Fr\'ed\'eric Bernicot, CNRS - Universit\'e de Nantes, Laboratoire Jean Leray, 2 rue de la Houssini\`ere, 44322 Nantes cedex 3, France\\}
\email{frederic.bernicot@univ-nantes.fr}

\address{Thierry Coulhon, Paris Sciences et Lettres Research University, 62 bis, rue Gay-Lussac, 75005 Paris, France}
\email{thierry.coulhon@univ-psl.fr}

\address{Dorothee Frey, Mathematical Sciences Institute, The Australian National University, Canberra ACT 0200, Australia}
\email{dorothee.frey@anu.edu.au}

\thanks{F. Bernicot's research is supported by ANR projects AFoMEN no. 2011-JS01-001-01 and HAB no. ANR-12-BS01-0013.\\
T. Coulhon's research was done while he was employed by the Australian National University. It was  supported by the Australian Research Council (ARC) grant DP  130101302.  \\
D. Frey's research is supported by the Australian Research Council (ARC) grants DP 110102488 and DP 120103692.}
\date{\today}

\begin{document}
\begin{abstract}  On  a  doubling metric measure space endowed with a ``carr\'e du champ'', we consider $L^p$ estimates $(G_p)$ of the gradient of the heat semigroup  and scale-invariant $L^p$ Poincar\'e inequalities $(P_p)$. We show that the combination of $(G_p)$ and  $(P_p)$  for $p\ge 2$ always implies two-sided Gaussian heat kernel bounds. The case $p=2$ is a famous theorem of Saloff-Coste, of which we give a shorter proof, without parabolic Moser iteration. We also give  a more direct proof of the  main result in \cite{HS}. This relies in particular on a new notion of $L^p$ H\"older regularity for a semigroup and on a  characterization of  $(P_2)$ in terms of harmonic functions.
\end{abstract}



\maketitle

\begin{quote}
\footnotesize\tableofcontents
\end{quote}

\section{Introduction}\label{intro}

\subsection{The Dirichlet form setting}

The reader who does not care about generality but is satisfied by a wide range of interesting examples can skip this section and in the rest of the paper think of $M$ as a complete Riemannian manifold satisfying the volume doubling property and $\mathcal{L}$ its (nonnegative) Laplace-Beltrami operator.  
However, the more general setting we are about to present covers many more examples, see \cite{GSC}.

Let $M$ be a locally compact separable metrisable space equipped with a Borel measure $\mu$, finite on compact sets and strictly positive on any non-empty open set.
For $\Omega$ a measurable subset of $M$, we shall often denote $\mu\left(\Omega\right)$ by $\left|\Omega\right|$.

Let $\mathcal{L}$ be a non-negative self-adjoint operator on $L^2(M,\mu)$ with dense domain ${\mathcal D}\subset L^2(M,\mu)$. Denote by $\mathcal{E}$ the associated quadratic form 
$$\mathcal{E}(f,g)=\int_M f \mathcal{L}g\,d\mu,$$
for $f,g\in \mathcal{D}$,  and by  $\mathcal{F}$ its domain, which contains ${\mathcal D}$.
Assume that $\mathcal{E}$ is a strongly local and regular Dirichlet form (see \cite{FOT, GSC} for precise definitions).  As a consequence, there exists  an energy measure $d\Gamma$, that is a signed measure depending in a bilinear way on $f,g\in{\mathcal F}$ such that
$$\mathcal{E}(f,g)=\int_M d\Gamma(f,g)$$ for all $f,g\in\mathcal{F}$. 
A possible definition of $d\Gamma$  is through the formula
\begin{equation}\label{def}
 \int \varphi\, d\Gamma(f,f)  = {\mathcal E}(\varphi f, f) - \frac{1}{2} {\mathcal E}(\varphi,f^2),
 \end{equation}
valid for $f\in \mathcal{F}\cap L^\infty(M,\mu)$ and $\varphi \in \mathcal{F}\cap \mathcal{C}_0(M)$.
Here $\mathcal{C}_0(M)$ denotes the space of continuous functions on $M$ that vanish at infinity. 
According to the Beurling-Deny-Le Jan formula, the energy measure satisfies a Leibniz rule, namely 
\begin{equation}
 d\Gamma(fg,h) = fd\Gamma(g,h) + g d\Gamma(f,h), \label{eq:leibniz}
\end{equation}
for all $f,g\in \mathcal{F}\cap L^\infty(M,\mu)$ and $h\in\mathcal{F}$, see \cite[Section 3.2]{FOT}.
One can define a pseudo-distance $d$ associated with $\mathcal{E}$ by
 \begin{equation}\label{defd}
d(x,y):=\sup\{f(x)-f(y) ; f\in \mathcal{F}\cap \mathcal{C}_0(M) \mbox{ s.t. }d\Gamma(f,f)\le d\mu\}.
\end{equation}
Throughout the whole paper, we assume that the pseudo-distance $d$ separates points, is finite everywhere, continuous and defines the initial topology of $M$ (see \cite{ST1} and \cite[Subsection 2.2.3]{GSC} for details).

When we are in the above situation, we shall say that $(M,d,\mu, {\mathcal E})$ is a metric  measure (strongly local and regular)  Dirichlet space. Note that this terminology is slightly abusive, in the sense that in the above presentation $d$ follows from ${\mathcal E}$.

For all $x \in M$ and all $r>0$, denote by $B(x,r)$ the open ball for the metric $d$ with centre $x$ and radius $r$, and  by $V(x,r)$ its measure $|B(x,r)|$.  For a ball $B$ of radius $r$ and  $\lambda>0$, denote by $\lambda B$   the ball concentric  with $B$ and with radius $\lambda r$.  We sometimes denote by $r(B)$ the radius of the ball $B$. Finally, we will use $u\lesssim v$ to say that there exists a constant $C$ (independent of the important parameters) such that $u\leq Cv$ and $u\simeq v$ to say that $u\lesssim v$ and $v\lesssim u$. 

We shall assume  that   $(M,d,\mu)$ satisfies the volume doubling property, that is
  \begin{equation}\label{d}\tag{$V\!D$}
     V(x,2r)\lesssim  V(x,r),\quad \forall~x \in M,~r > 0.
    \end{equation}
It follows that there exists  $\nu>0$  such that
     \begin{equation*}\label{dnu}\tag{$V\!D_\nu$}
      V(x,r)\lesssim \left(\frac{r}{s}\right)^{\nu} V(x,s),\quad \forall~ x \in M,~r \ge s>0,
    \end{equation*}
which implies
     \begin{equation*}
      V(x,r)\lesssim \left(\frac{d(x,y)+r}{s}\right)^{\nu} V(y,s),\quad \forall~ x,y \in M, ~r \ge s>0.
    \end{equation*}
An easy consequence of \eqref{d} is that  balls with a non-empty intersection and comparable radii have comparable measures.

We shall say that $(M,d,\mu, {\mathcal E})$ is a doubling metric  measure   Dirichlet space if it is a metric measure space  endowed with a strongly local and regular Dirichlet form  and satisfying \eqref{d}.

\subsection{Heat kernel estimates}
Let  $(M,d,\mu, {\mathcal E})$ be a doubling metric  measure   Dirichlet space.
The Dirichlet form ${\mathcal E}$ gives rise to a strongly continuous  semigroup  $(e^{-t\mathcal{L}})_{t>0}$ of self-adjoint contractions on $L^2(M,\mu)$. In addition $(e^{-t\mathcal{L}})_{t>0}$  is submarkovian, that is
$0\leq e^{-t\mathcal{L}}f\leq 1$ if  $0\le f\leq 1$. It follows that the semigroup $(e^{-t\mathcal{L}})_{t>0}$ is uniformly  bounded on $L^p(M,\mu)$ for $p\in[1,+\infty]$. Also, $(e^{-t\mathcal{L}})_{t>0}$ is   bounded analytic on $L^p(M,\mu)$ for $1<p<+\infty$  (see \cite{topics}), which means  that  $(t\mathcal{L}e^{-t\mathcal{L}})_{t>0}$ is bounded on $L^p(M,\mu)$ uniformly in $t>0$.
Moreover, due to the doubling property \eqref{d}, the semigroup has the conservation property (see \cite{Grigo, ST1}),  that is
$$ e^{-t\mathcal{L}}1=1 ,\qquad \forall\, t>0.$$
In particular, the above assumptions rule out the case of a non-empty boundary with Dirichlet boundary conditions, see the comments
in \cite[pp. 13--14]{GSC}.

Under our assumptions,  according to \cite[Theorems 7.2.1-7.2.2]{FOT},  $e^{-t\mathcal{L}}$ has a  kernel, that is for all $t>0$ a measurable function $p_t:M\times M\to\R_+$ such that
$$e^{-t\mathcal{L}}f(x)=\int_Mp_t(x,y)f(y)\,d\mu(y), \quad \mbox{a.e. }x\in M.$$  We shall call $p_t$ the heat kernel associated with $\mathcal{L}$ (in fact with $(M,d,\mu,\mathcal{E}))$. Then $p_t(x,y)$ is nonnegative and symmetric in $x,y$  since $e^{-t\mathcal{L}}$ is positivity preserving and self-adjoint for all $t>0$.

One may  naturally ask for upper and lower estimates of $p_t$ (for upper estimates, see for instance the recent article \cite{BCS} and the many relevant references therein; for lower estimates, we will give more references below). A typical upper estimate is
  \begin{equation}\tag{$DU\!E$}
 p_{t}(x,y)\lesssim
\frac{1}{\sqrt{V(x,\sqrt{t})V(y,\sqrt{t})}}, \quad \forall~t>0,\,\mbox{a.e. }x,y\in
 M.\label{due}
\end{equation}
This estimate is called on-diagonal because if $p_t$ happens to be continuous then \eqref{due} can be rewritten as
  \begin{equation}
p_{t}(x,x)\lesssim
\frac{1}{V(x,\sqrt{t})}, \quad \forall~t>0,\,\forall\,x\in
 M.
\end{equation}

Under  \eqref{d}, \eqref{due} self-improves into a Gaussian upper estimate (see \cite[Theorem 1.1]{Gr1} for the Riemannian case,   \cite[Section 4.2]{CS} for a metric measure space setting):
\begin{equation}\tag{$U\!E$}
p_{t}(x,y)\lesssim
\frac{1}{V(x,\sqrt{t})}\exp
\left(-\frac{d^{2}(x,y)}{Ct}\right), \quad \forall~t>0,\, \mbox{a.e. }x,y\in
 M.\label{UE}
\end{equation}

The proof of this fact in \cite[Section 4.2]{CS} relies on the following  Davies-Gaffney estimate which was proved in our setting in \cite{ST2}:
 for all open subsets $E,F \subset M$,  $f\in L^2(M,\mu)$ supported in $E$, and $t>0$, 
\begin{equation} \left(\int_{F} |e^{-t\mathcal{L}} f |^2 d\mu \right)^{1/2}   \leq e^{- \frac{d^2(E,F)}{4t}} \left( \int_E |f|^2 d\mu \right)^{1/2}, \label{eq:DG0} \end{equation}
where $d(E,F)$ denotes the distance between $E$ and $F$. 
For more on the Davies-Gaffney estimate, see for instance \cite[Section 3]{CS}.

It is well-known on the contrary that  the matching Gaussian lower bound 
\begin{equation}\tag{$L\hspace{-.3mm}E$}
 p_{t}(x,y)\gtrsim \frac{1}{V(x,\sqrt{t})}\exp\left(-\frac{d^{2}(x,y)}{ct}\right) ,\quad \forall~t>0, \ \,\mbox{a.e. }x,y\in
 M\label{LE}
\end{equation}
does not always follow from \eqref{due}  (see \cite{BCF2}). Conversely, under \eqref{d}, \eqref{LE} implies \eqref{UE} (see \cite[Theorem 1.3]{Bou0} and \cite{CS3}).

It is not too difficult to prove in our situation that the conjunction of the upper and lower bounds \eqref{UE} and \eqref{LE} (that is, \eqref{d} and  \eqref{LE})  is equivalent to a uniform parabolic Harnack inequality, see \cite{FS} as well as \cite[Section 1]{BGK}.
One also knows (\cite[Theorem 2.32]{GSC}) that this Harnack inequality self-improves into  a H\"older regularity estimate for the heat kernel: there exists $\eta\in(0,1]$
\begin{equation}\label{ho}\tag{$H^\eta$}
\left|p_t(x,z)-p_t(y,z)\right| \lesssim \left(\frac{d(x,y)}{\sqrt{t}}\right)^\eta p_t(x,z),
\end{equation}
for all $t>0$ and a.e. $x,y,z\in
 M$ such that $d(x,y)\leq \sqrt{t}$.
 Note that, if \eqref{ho} holds for some $\eta\in (0,1]$, $p_t$ admits in particular a continuous version.

 What is also true, but much more difficult to prove (see  \cite{Gr0}, \cite{S}, \cite{SU}, \cite{parma}, \cite{SA}, \cite{ST3}, as well as \cite[Theorem 2.31]{GSC}) is that  $\eqref{UE}+\eqref{LE}$ is also equivalent  to \eqref{d} together with the following scale-invariant Poincar\'e inequality $(P_2)$:
 \begin{equation}\tag{$P_2$} 
 \left( \aver{B} \left| f - \aver{B} f d\mu \right|^2 d\mu \right)^{1/2} \lesssim r \left(\int_B d\Gamma(f,f)\right)^{1/2},  \label{P2}
\end{equation}
for every  $f\in {\mathcal D}$ and every ball $B\subset M$ with radius $r$. Here $\aver{B} f d\mu=\frac{1}{|B|} \int_B f d\mu$ denotes the average of $f$ on $B$.
A somewhat simplified proof of the main implication, namely the one from  $\eqref{d}+(P_2)$ to $\eqref{UE}+\eqref{LE}$,  has been given in \cite{HS}.
One of the outcomes of the present article will be to provide  a further simplification (see  the proof of Theorem \ref{rem:ep} below) as well as a short proof of the main result in \cite{HS} (see Theorem \ref{thm:2} below).
 \medskip

\subsection{$L^p$ Poincar\'e inequalities and gradient estimates}
The above scale-invariant Poincar\'e inequality $(P_2)$ quantifies the control of the oscillation of functions by the Dirichlet form. 
As it happens in many instances in analysis, it is important to have at hand a full scale of conditions for $p\in [1,+\infty]$, not just $p=2$. This requires to have, beyond the notion of $L^2$ norm of the gradient provided by the Dirichlet form,  a notion of $L^p$ norm of the gradient, hence a pointwise notion of    length of the gradient.

The relevant notion in our general setting is the one of
 ``carr\'e du champ'' (see for instance \cite{GSC} and the references therein). The Dirichlet form (or its energy measure) admits a ``carr\'e du champ'' if for all $f,g\in\mathcal{F}$ the energy measure $d\Gamma(f,g)$ is absolutely continuous with respect to $\mu$. Then the  density $\Upsilon(f,g) \in L^1(M,\mu)$ of $d\Gamma(f,g)$ is called the ``carr\'e du champ'' and satisfies the following inequality
 \begin{equation}   |\Upsilon(f,g)|^2 \leq \Upsilon(f,f) \Upsilon(g,g).  \label{eq:carre} \end{equation} In the sequel, when we assume that $(M,d,\mu,\mathcal{E})$ admits
 a ``carr\'e du champ'',  we shall abusively denote $\left[\Upsilon(f,f)\right]^{1/2}$ by $ |\nabla f|$.
 This has the advantage to stick to the more intuitive and classical Riemannian notation, but one should not forget that one works 
 in a much more general setting (see for instance \cite{GSC} for examples), and that one never uses differential calculus in the classical sense.
 
 We shall summarise this situation by saying that $(M,d,\mu, {\mathcal E})$ is a  metric  measure   Dirichlet space with a ``carr\'e du champ''.
 
We can now formulate the $L^p$ versions of the scale-invariant Poincar\'e inequalities, which may or may not be true and, contrary to the   H\"older regularity conditions for the heat semigroup,  do depend on $p\in[1,+\infty)$. More precisely, for $p\in[1,+\infty)$, one says that  $(P_p)$ holds if
\begin{equation}\tag{$P_p$}
 \left( \aver{B} \left| f - \aver{B} f d\mu \right|^p d\mu \right)^{1/p} \lesssim r \left(\aver{B} |\nabla f|^p d\mu \right)^{1/p}, \qquad \forall\, f\in {\mathcal F}, \label{Pp}
\end{equation}
where $B$ ranges over balls in $M$ of radius $r$.
Recall that $(P_p)$ is weaker and weaker as $p$ increases, that is $(P_p)$ implies $(P_q)$ for $q>p$, see for instance \cite{HaKo}, and the $p=\infty$ version is trivial in the Riemannian setting  (see however interesting developments for more general metric measure spaces in \cite{DCJS}).

On the Euclidean space, $(P_p)$ holds for all $p\in[1,+\infty]$. On  the connected sum of two copies of  $\R^n$, $(P_p)$ is valid if and only if $p>n$, as one can see by adapting the proof of  \cite[Example 4.2]{HaKo}. More interesting examples follow from \cite[Theorem 6.15]{HK}, see also \cite[Section 5]{DCJS1}. On conical manifolds with compact basis, $(P_p)$ holds at least for $p\ge 2$ (see \cite{CHQL}). A deep result from \cite{KZ} states if that $(P_p)$ holds for some $p\in(1,+\infty)$,
then $(P_{p-\eps})$ holds for some $\eps>0$. Finally the set $\{p\in [1,+\infty]; (P_p)\mbox{ holds on }M\}$ may be either $\{+\infty\}$, or $[1,+\infty]$, or of the form $(p_M,+\infty]$ for some $p_M>1$.

\bigskip

We will also consider estimates on the gradient (or ``carr\'e du champ") of the semigroup, which were introduced in  \cite{ACDH}, namely, for $p\in[1,+\infty]$, 
\begin{equation} \label{Gp}
\sup_{t>0} \|\sqrt{t}|\nabla e^{-t\mathcal{L}} |\|_{p\to p} <+\infty \tag{$G_p$},
\end{equation}
which is equivalent to the interpolation inequality
\begin{equation} \label{mult}  \||\nabla f|\|_p^2 \lesssim \|\mathcal{L} f\|_p\| f\|_p, \qquad \forall\,f\in {\mathcal D}
\end{equation}
(see \cite[Proposition 3.6]{CS2}).
 As far as examples are concerned, $(G_\infty)$ 
 holds on manifolds with non-negative Ricci curvature (\cite{LY}), Lie groups with polynomial volume growth (\cite{saloffpoly}),
 and co-compact covering manifolds  with polynomial growth deck transformation group (\cite{Du0}, \cite{Du}).
On the other hand,  conical manifolds with a compact basis  provide a family of doubling spaces $(M,d,\mu,\mathcal{E})$ with a ``carr\'e du champ'' satisfying \eqref{UE} and \eqref{LE} such that for every  $p_0>2$ there exist examples  in this family  where  $(G_p)$  holds for  $1<p<p_0$ and not for  $p\ge p_0$, see \cite{HQL1},\cite{HQL2},\cite{CHQL}.


\subsection{Our main results}

In the present paper, we are going to look at the combination $(G_p)+(P_p)$ for $1\le p\le +\infty$, and especially for $2\le p<+\infty$, on a  doubling metric  measure   Dirichlet space with a ``carr\'e du champ''. 
For $1\le p\le 2$, $(G_p)+(P_p)$ is nothing but $(P_p)$ and therefore is weaker and weaker as $p$ goes from $1$ to $2$.  On the contrary, for $2\le p\le +\infty$, since  $(G_p)$ is stronger as $p$ increases, whereas  $(P_p)$ is weaker, $(G_p)+(P_p)$ does not exhibit a priori any monotonicity. 

At one end of the range, $(G_\infty)+(P_\infty)$, at least in the Riemannian setting, is nothing but $(G_\infty)$, which does not seem to have consequences in itself. 
 However, it has been  shown in \cite[Corollary 2.2]{CS2} that, in presence of \eqref{d} and \eqref{due}, $(G_\infty)$ implies \eqref{LE}.
At the  other end of the range, for $p=2$, we already recalled the fundamental fact 
that $(G_2)+(P_2)=(P_2)$ implies $\eqref{UE}+\eqref{LE}$.

We are going to complete the  picture for $2<p<+\infty$ and by the same token  simplify  the proof of the case $p=2$.

 First,  relying on a variant of the Davies-Gaffney estimate that holds in our setting, we will prove in Section \ref{sec:ue}  that for $2\le p< +\infty$, the combination $(G_p)+(P_p)$  
 implies the upper estimate \eqref{due} and therefore \eqref{UE}.
 
 Then, in Section \ref{sec:oscillation}, we shall study the transition from \eqref{UE} to \eqref{LE} in the spirit of 
\cite{coupre} and \cite{Boutayeb}. The main novelty here will be a notion of $L^p$ $\eta$-H\"older regularity of the heat kernel: for $p\in[1,+\infty]$ and $\eta\in(0,1]$, we shall say that property \eqref{gp} holds if for every   $0<r\leq \sqrt{t}$, every pair of  concentric balls $B_r$, $B_{\sqrt{t}}$ with respective radii $r$ and $\sqrt{t}$, and every function $f\in L^{p}(M,\mu)$,
\begin{equation}
\left(\aver{B_r} \left| e^{-t\mathcal{L}}f - \aver{B_r}e^{-t\mathcal{L}}f d\mu \right|^{p} d\mu \right)^{1/p} \lesssim \left( \frac{r}{\sqrt{t}} \right)^\eta  \left|B_{\sqrt{t}}\right|^{-1/p} \|f\|_p
 \tag{$H_{p,p}^\eta$},
\end{equation}
with the obvious modification for $p=+\infty$.

Crucial to our approach is Theorem \ref{thm} below where we prove the equivalence, under \eqref{d} and \eqref{UE}, between  the Gaussian lower bound \eqref{LE} and the existence of some $p\in[1,+\infty)$ and $\eta>0$ such that \eqref{gp} holds, a property which turns out to be independent of $p\in[1,+\infty)$. As a consequence, we shall give a quick proof of the fact that  under \eqref{dnu},  $(G_p)+(P_p)$ for $p>\nu$ implies $\eqref{UE}+\eqref{LE}$, which had been proved in \cite[Thm. 5.2]{coupre} in the polynomial volume growth case.  We also recover the main result of  \cite{Boutayeb}. 
   Later, we shall see that the limitation  $p>\nu$ in the above implication is artificial. But the interest of this first approach is that it   does not require any kind of Moser iteration.
   
   In Section \ref{section:2}, we turn to the case $p=2$. We give a simple proof of the implication from $\eqref{d}+(P_2)$ to $\eqref{UE}+\eqref{LE}$; the only remaining non-trivial part is the implication from $\eqref{d}+(P_2)$ to the most classical De Giorgi property $(DG_2)$, which is recalled in Appendix \ref{AppDG}. With similar arguments, we also obtain a new proof of the result from \cite{HS} that the elliptic regularity  together with a scale-invariant local Sobolev  inequality imply the parabolic Harnack inequality (Theorem \ref{thm:2}).

Section \ref{section:DG} is devoted to considerations on $L^p$ versions of  De Giorgi property and  Caccioppoli inequality, which may be of independent interest and are used anyway in the sequel.
 
Then Section \ref{sec:poi} is devoted to the proof of Poincar\'e inequality $(P_2)$ under  $(G_p)+(P_p)$ for any $p\in(2,+\infty)$. It relies on on Proposition \ref{prop1} and on  Appendix \ref{App:ste} where a self-improving property of reverse H\"older estimates is spelled out.

Sections \ref{section:2} and \ref{sec:poi} together yield the following statement, which summarises
most of our results:
\begin{theorem} \label{th2} Let $(M,d,\mu,\mathcal{E})$ be a doubling metric measure Dirichlet space with a ``carr\'e du champ". Assume $(G_{p_0})$ and $(P_{p_0})$ for some $p_0\in [2,+\infty)$. Then 
 the heat kernel satisfies the two-sided Gaussian estimates \eqref{UE} and \eqref{LE}.
\end{theorem}

One can also formulate the above result by saying that  in our setting  $(G_{p_0})+(P_{p_0})$ for  $p_0>2$ implies $(P_{2})$. Finally $(G_{p})+(P_{p})$ is stronger and stronger as $p$ increases above $2$. Theorem \ref{th2} can be combined with the main result in \cite{ACDH} and yield results on Riesz transforms.

Since our results avoid parabolic Moser iteration, which is very hard to run directly in a discrete time setting (see \cite{Del}), they
are well suited to an extension to random walks on discrete graphs. As a matter of fact, our Appendix \ref{AppDG} is inspired by   \cite{AC}, but on the other hand our approach below gives a simpler proof of the main result in \cite{AC} by avoiding  the iteration step in \cite[Proposition 4.5]{AC}. 


\section{From Poincar\'e and gradient estimates to  heat kernel upper  bounds}
\label{sec:ue}

In this section we shall need a version of the Davies-Gaffney estimate \eqref{eq:DG0} which also includes the gradient, namely
\begin{equation} \left(\int_{F} |e^{-t\mathcal{L}} f |^2 d\mu \right)^{1/2}  + \sqrt{t} \left(\int_{F} |\nabla e^{-t\mathcal{L}}f|^2\,d\mu \right)^{1/2} \lesssim e^{-c \frac{d(E,F)^2}{t}} \left( \int_E |f|^2 d\mu \right)^{1/2}, \label{eq:DG} \end{equation}
for some $c>0$, all open subsets $E,F \subset M$,  $f\in L^2(M,\mu)$ supported in $E$, and $t>0$, $d(E,F)$ being the distance between $E$ and $F$. 
The proof of  this fact in \cite[Section 3.1]{ACDH} works in our setting of a Dirichlet space with a ``carr\'e du champ". Indeed, the proof relies on the following inequality: for $\varphi$ a non-negative cut-off function with  support $S$,
\begin{align} \label{eq:youpi}\int \varphi |\nabla e^{-t\mathcal{L}} f|^2 d\mu \leq &  \left(\int  |e^{-t\mathcal{L}} f|^2 |\nabla \varphi|^2 d\mu\right)^{1/2} \left(\int_S  |\nabla e^{-t\mathcal{L}} f|^2 d\mu\right)^{1/2}\\ \nonumber&+\int \varphi |e^{-t\mathcal{L}}f| | \mathcal{L}e^{-t\mathcal{L}} f| d\mu,  \end{align}
which follows from  \eqref{def}, \eqref{eq:leibniz}, and \eqref{eq:carre}. 

\begin{proposition} \label{prop:ppgp} Let $(M,d,\mu,\mathcal{E})$ be a metric measure Dirichlet space with a ``carr\'e du champ" satisfying  \eqref{d}.   Then the combination of  $(G_{p})$ with $(P_{p})$ for some $p\in[2,+\infty)$ implies \eqref{UE}.
\end{proposition}

\begin{proof} Assume first $2<p< +\infty$. From the self-improving property of $(P_{p})$  (see \cite{KZ}), there exists $\tilde p\in(2,p)$ such that $(P_{\tilde p})$ holds. Then, by interpolating between the $L^2$ Davies-Gaffney estimate  for $\nabla e^{-t\mathcal{L}}$ contained in \eqref{eq:DG} and $(G_{p})$, one obtains that for $t>0$ the operator $\sqrt{t} \nabla e^{-t\mathcal{L}}$ satisfies $L^{\tilde p}$-$L^{\tilde p}$ off-diagonal estimates at the scale $\sqrt{t}$. Similarly, by interpolating the uniform $L^\infty$ boundedness with \eqref{eq:DG}, one sees that the semigroup $e^{-t\mathcal{L}}$ also satisfies such estimates.
Namely, for some $c>0$,
\begin{equation} \|\sqrt{t}|\nabla e^{-t\mathcal{L}}| \|_{L^{\tilde p}(B) \to L^{\tilde p} (\tilde B)} + \| e^{-t\mathcal{L}} \|_{L^{\tilde p}(B) \to L^{\tilde p} (\tilde B)} \lesssim \exp\left(-c\frac{d^2(B,\tilde B)}{t}\right),\label{eq:od}
\end{equation}
for every $t>0$ and all balls $B,\tilde B$.
On the other hand, the $(P_{\tilde p})$ Poincar\'e inequality self-improves into a $(P_{{\tilde p},q})$ inequality for some $q>{\tilde p}$ (given by $q^{-1}={\tilde p}^{-1}-\nu^{-1}$ if ${\tilde p}<\nu$, $q=+\infty$ if ${\tilde p}<\nu$ and any $q>\nu$ if ${\tilde p}=\nu$, see \cite{FPW}). That is, for every ball $\tilde B$ of radius $\sqrt{t}$, one has
$$ \left(\aver{\tilde B} \left| f - \aver{\tilde B}  f d\mu \right|^q d\mu \right)^{1/q} \lesssim \sqrt{t}\left(\aver{\tilde B} \left| \nabla  f \right|^{\tilde p} d\mu \right)^{1/{\tilde p}}.$$
Hence
$$ \left(\aver{\tilde B} \left| e^{-t\mathcal{L}}f - \aver{\tilde B} e^{-t\mathcal{L}} f d\mu \right|^q d\mu \right)^{1/q} \lesssim \left(\aver{\tilde B} \left|\sqrt{t} \nabla e^{-t\mathcal{L}} f \right|^{\tilde p} d\mu \right)^{1/{\tilde p}},$$
for all $t>0$ and $f\in L^{\tilde p}(M,\mu)$.
It follows by Jensen's inequality that
$$  \left(\aver{\tilde B} \left| e^{-t\mathcal{L}}f \right|^q d\mu \right)^{1/q} \lesssim  \left(\aver{\tilde B} \left| e^{-t\mathcal{L}}f \right|^{\tilde p} d\mu \right)^{1/{\tilde p}} + \left(\aver{\tilde B} \left|\sqrt{t} \nabla e^{-t\mathcal{L}} f \right|^{\tilde p} d\mu \right)^{1/{\tilde p}}.$$
Then from \eqref{eq:od}, we deduce that for every pair of balls $B,\tilde B$ of radius $\sqrt{t}$ one has
\begin{equation} \| e^{-t\mathcal{L}} \|_{L^{\tilde p}(B) \to L^q (\tilde B)} \lesssim \exp\left(-c\frac{d^2(B,\tilde B)}{t}\right) |\tilde B|^{\frac{1}{q}-\frac{1}{{\tilde p}}}. \label{eq:od2} \end{equation}
We now use \cite{BCS}, and refer to it for more details. Set  $V_{r}(x):=V(x,r)$,  and denote abusively by $w$ the operator of multiplication by a function $w$. Using doubling,  \eqref{eq:od2} may be written as
$$ \| V_{\sqrt{t}}^{\frac{1}{{\tilde p}}-\frac{1}{q}} e^{-t\mathcal{L}} \|_{L^{\tilde p}(B) \to L^q (\tilde B)} \lesssim \exp\left(-c\frac{d^2(B,\tilde B)}{t}\right).$$
By the doubling property, we may sum this inequality over a covering of the whole space at the scale $\sqrt{t}$ and  deduce $(VE_{{\tilde p},q})$, which is 
$$\sup_{t>0} \| V_{\sqrt{t}}^{\frac{1}{{\tilde p}}-\frac{1}{q}} e^{-t\mathcal{L}} \|_{{\tilde p}\to q} <+\infty.$$
By duality,  one obtains $(EV_{q',{\tilde p}'})$,
that is
$$\sup_{t>0} \|  e^{-t\mathcal{L}}V_{\sqrt{t}}^{\frac{1}{{\tilde p}}-\frac{1}{q}} \|_{q'\to {\tilde p}'} <+\infty.$$  Then by interpolation \cite[Proposition 2.1.5]{BCS} between $(VE_{{\tilde p},q})$ and $(EV_{q',{\tilde p}'})$, one obtains $(VEV_{r,r',\frac{1}{r}-\frac{1}{2}})$,
that is
$$\sup_{t>0} \|  V_{\sqrt{t}}^{\frac{1}{r}-\frac{1}{2}}e^{-t\mathcal{L}}V_{\sqrt{t}}^{\frac{1}{r}-\frac{1}{2}} \|_{r\to r'} <+\infty,$$ where $1\le r<2$ is given by $\frac{1}{r}=\frac{1}{2}(\frac{1}{{\tilde p}}+\frac{1}{q'})=\frac{1}{2}+(\frac{1}{{\tilde p}}-\frac{1}{q})$. Then  $(EV_{r,2})$ holds by \cite[Remark 2.1.3]{BCS}.  Thanks to the $L^1$-uniform boundedness of the semigroup,  the extrapolation \cite[Proposition 4.1.9]{BCS}  yields $(EV_{1,2})$, hence \eqref{due} by \cite[Proposition 2.1.2]{BCS} and   \eqref{UE} by \cite[Section 4.2]{CS}.

Finally, if $p=2$,  one can run the above proof by setting directly ${\tilde p}=2$.  Alternatively, one can see by \cite[Section 5]{CS3} that $(P_2)$ and \eqref{d} imply the so-called Nash inequality $(N)$ and apply \cite[Theorem 1.2.1]{BCS}. 
\end{proof}


\begin{rem} The case $1\le p<2$ of Proposition $\ref{prop:ppgp}$  follows trivially from the case $p=2$.
\end{rem}

\begin{rem} One may avoid the use of the highly non-trivial result from {\rm \cite{KZ}} by  assuming directly  $(G_{p})$  and $(P_{q})$ for some $q\in(2,p)$.
\end{rem}


\section{$L^p$ H\"older regularity of the heat semigroup and heat kernel lower bounds} \label{sec:oscillation}

The following statement is valid in a more general setting than the one presented in Section \ref{intro} and used in Section \ref{sec:ue}: it is enough to consider a metric measure space $(M,d,\mu)$ satisfying $(VD)$, endowed with a  semigroup  $(e^{-t\mathcal{L}})_{t>0}$ acting on $L^p(M,\mu)$, $1\le p\le +\infty$.
For $1\le p\le +\infty$ let us write the $L^p$-oscillation for $u\in L^p_{loc}(M,\mu)$ and $B$ a ball:
$$ p\text{-}\osc_B(f) := \left(\aver{B} |f-\aver{B} f \,d\mu|^p \,d\mu\right)^{1/p},$$
if $p<+\infty$ and
$$ \infty\text{-}\osc_B(f) := \esssup_{B}| f - \aver{B}f\,d\mu |.$$

\begin{proposition} \label{proposition} Let $(M,d,\mu,\mathcal{L})$ as above. Let $p\in[1,+\infty]$ and $\eta\in(0,1]$.  Then the following two conditions are equivalent: 
\begin{enumerate}
\item for all   $0<r\leq \sqrt{t}$, every pair of concentric balls $B_r$, $B_{\sqrt{t}}$ with respective radii $r$ and $\sqrt{t}$, and every function $f\in L^{p}(M,\mu)$,
\begin{equation} \label{gp}
p\text{-}\osc_{B_r}(e^{-t\mathcal{L}}f) \lesssim \left( \frac{r}{\sqrt{t}} \right)^\eta  \left|B_{\sqrt{t}}\right|^{-1/p} \|f\|_p
 \tag{$H_{p,p}^\eta$}.
\end{equation}

\item for all   $0<r\leq \sqrt{t}$, every pair of concentric balls $B_r$, $B_{\sqrt{t}}$ with respective radii $r$ and $\sqrt{t}$, and every function $f\in L^{p}(M,\mu)$,
\begin{equation}
 \esssup_{x,y \in B_r} \left| e^{-t\mathcal{L}}f(x)-e^{-t\mathcal{L}}f(y) \right| \lesssim \left( \frac{r}{\sqrt{t}} \right)^{\eta}   \left|B_{\sqrt{t}}\right|^{-1/p} \|f\|_p. \label{eq:am} \tag{$H_{p,\infty}^\eta$}
\end{equation}
\end{enumerate}
\end{proposition}

\begin{rem}\label{name} It is easy to see that \eqref{eq:am} is equivalent to the following condition, which justifies its name: for all  $0<r\leq \sqrt{t}$, every pair of concentric balls $B_r$, $B_{\sqrt{t}}$ with respective radii $r$ and $\sqrt{t}$, and every function $f\in L^{p}(M,\mu)$,
\begin{equation}
\infty\text{-}\osc_{B_r}(e^{-t\mathcal{L}}f) \lesssim \left( \frac{r}{\sqrt{t}} \right)^{\eta}  \left|B_{\sqrt{t}}\right|^{-1/p} \|f\|_p. \label{gp2}
 \end{equation}
\end{rem}

 Proposition \ref{proposition} is an easy consequence of a well-known  characterisation of H\"older continuous functions in terms of the growth of their $L^p$ oscillations on balls. This  result is due to Meyers {\rm\cite{meyers}} in  the Euclidean space, and its proof was later simplified, see e.g. {\rm \cite[III.1]{Gia0}}. It can be formulated in terms of embeddings of Morrey-Campanato spaces into H\"older spaces. The proof goes through in a doubling metric measure space setting (see {\rm \cite[Proposition 2.6]{AC2}} for an $L^2$ version).  
 We give a proof for the sake of completeness.
 
 \begin{lemma}\label{lemma:meyers} Let $(M,d,\mu)$ be a metric  measure    space satisfying \eqref{d}. Let $1\le p<+\infty$ and $\eta>0$.  Then  for every function $f\in L^p_{loc}(M,\mu)$ and every ball $B$ in $(M,d,\mu)$, 
\begin{equation*} \|f\|_{C^\eta(B)} := \esssup_{\genfrac{}{}{0pt}{}{x,y\in B}{x\neq y}} \frac{|f(x)-f(y)|}{d^\eta(x,y)} 
 \lesssim  \|f\|_{C^{\eta,p}(B)}:=\sup_{\tilde B \subset 6B}  \frac{p\text{-}\osc_{\tilde B} (f)}{r^\eta(\tilde{B})}.
\end{equation*}

\end{lemma}
\begin{proof} Let $x,y\in B$ be  Lebesgue points for $f$. Let $B_i(x)=B(x,2^{-i}d(x,y))$, for $i\in\N$. Note that for all $i\in \N$,  $B_i(x)\subset B_0(x)\subset 3B$.
Write
\begin{eqnarray*}
\left|f(x) - \aver{B_0(x)} f d\mu \right|&\le&\sum_{i\geq 0}\left|\aver{B_i(x)} f d\mu - \aver{B_{i+1}(x)} f d\mu \right|\\
&\le&\sum_{i\geq 0}\aver{B_{i+1}(x)}\left|f-\aver{B_{i}(x)} f d\mu  \right|\,d\mu\\
&\le&\sum_{i\geq 0}\left(\aver{B_{i+1}(x)}\left|f-\aver{B_{i}(x)} f d\mu  \right|^p\,d\mu\right)^{1/p}\\
&\lesssim & \sum_{i\geq 0}p\text{-}\osc_{B_i(x)} (f),
\end{eqnarray*}
where the last inequality uses doubling.
It follows that
\begin{eqnarray*}
\left|f(x) - \aver{B_0(x)} f d\mu \right|
&\leq& \left(\sum_{i\geq 0} \left[r\left(B_i(x)\right)\right]^\eta\right) \|f\|_{C^{\eta,p}(B)}\\
&=& \left(\sum_{i\geq 0} 2^{-\eta i}d^\eta(x,y)\right) \|f\|_{C^{\eta,p}(B)}\\
&\lesssim& d^\eta(x,y) \|f\|_{C^{\eta,p}(\tilde{B})},
\end{eqnarray*}
as well as the similar estimate with the roles of $x,y$ exchanged.
Finally, since  $B_0(y),B_0(x) \subset 2B_0(x)$  with comparable measures by doubling,
\begin{eqnarray*}
\left|\aver{B_0(x)} f d\mu - \aver{B_0(y)} f d\mu \right|&\le& \left|\aver{B_0(x)} f d\mu - \aver{2B_0(x)} f d\mu \right|+\left|\aver{B_0(y)} f d\mu - \aver{2B_0(x)} f d\mu \right|\\&\lesssim& \aver{2B_0(x)} \left|f - \aver{2B_0(x)} f d\mu \right|\,d\mu\\&\leq& \left(\aver{2B_0(x)} \left|f - \aver{2B_0(x)} f d\mu \right|^p\,d\mu\right)^{1/p}\\
&= & p\text{-}\osc_{2B_0(x)} (f),
\end{eqnarray*}
hence
$$
\left|\aver{B_0(x)} f d\mu - \aver{B_0(y)} f d\mu \right|\leq d^\eta(x,y) \|f\|_{C^{\eta,p}(\tilde{B})}.
$$
The claim follows by writing
$$|f(x)-f(y)|\le \left|f(x) - \aver{B_0(x)} f d\mu \right|+\left|\aver{B_0(x)} f d\mu - \aver{B_0(y)} f d\mu \right|+\left|f(y) - \aver{B_0(y)} f d\mu \right|.$$
\end{proof}

\begin{proof}[{Proof of Proposition $\ref{proposition}$}] 
The  implication from  \eqref{eq:am} to \eqref{gp} is obvious by integration.  The case $p=+\infty$ of  the converse is similar to Remark \ref{name}.

Assume   \eqref{gp} for $1\le p<+\infty$. Let $t>0$, and  $B_r$ a ball of radius  $0<r\leq \sqrt{t}$.
From   Lemma \ref{lemma:meyers}, we deduce that for $f\in L^p(M,\mu)$, a.e. $x,y\in B_r$, 
\begin{equation}\label{eq:ouffff}
 |e^{-t\mathcal{L}}f(x)-e^{-t\mathcal{L}}f(y)|
    \lesssim r^{\eta}  \sup_{\tilde B \subset 6B_r}  \frac{p\text{-}\osc_{\tilde B} (e^{-t\mathcal{L}}f)}{r^\eta(\tilde{B})}. 
 \end{equation}
Now \eqref{gp} yields
\begin{equation*}
\frac{p\text{-}\osc_{\tilde B} (e^{-t\mathcal{L}}f)}{r^\eta(\tilde{B})}  \lesssim t^{-\eta/2} |\tilde B_{\sqrt{t}}|^{-1/p} \|f\|_p,
\end{equation*}
where $\tilde B_{\sqrt{t}}$ is the ball concentric to $\tilde B$ with radius $\sqrt{t}$. The balls $\tilde B_{\sqrt{t}}$ and $B_{\sqrt{t}}$ have the same radius 
and, if $\tilde B \subset 6B_r$, it follows  by doubling that $|\tilde B_{\sqrt{t}}|$ and $|B_{\sqrt{t}}|$ are comparable, hence
\begin{equation}\label{roger}
\sup_{\tilde B \subset 6B_r}\frac{p\text{-}\osc_{\tilde B} (e^{-t\mathcal{L}}f)}{r^\eta(\tilde{B})}  \lesssim t^{-\eta/2} |B_{\sqrt{t}}|^{-1/p} \|f\|_p, 
\end{equation}
 and \eqref{eq:ouffff} together with \eqref{roger} yield \eqref{eq:am}.
\end{proof}

Following some ideas in \cite[Theorem 3.1]{coupre}, we can now identify   \eqref{gp} as the property needed to pass from \eqref{UE} to  \eqref{LE}.

\begin{theorem} \label{thm} Let $(M,d,\mu,\mathcal{E})$  be a metric measure Dirichlet space satisfying \eqref{d} and the upper Gaussian estimate \eqref{UE}. If there exist $p\in[1,+\infty]$ and $\eta\in(0,1]$ such that \eqref{gp} is satisfied,
then  the  Gaussian lower bound \eqref{LE} holds. Conversely \eqref{LE}  implies \eqref{gp} for all $p\in[1,+\infty)$ and some $\eta\in(0,1]$.
\end{theorem}

\begin{rem} Let us emphasise two by-products of Theorem $\ref{thm}$:
 \begin{itemize} 
  \item  \eqref{LE} is equivalent to the existence of some $p\in[1,+\infty)$ and some $\eta\in(0,1]$ such that \eqref{gp} holds;
  \item  The property ``there exists $\eta>0$ such that $(H_{p,p}^\eta)$ holds'' is independent of $p\in[1,+\infty)$. 
 One can in fact prove that the property $(H_{p,p}^\eta)$ itself is $p$-independent of $p\in[1,+\infty]$, up to an arbitrarily small loss on $\eta$.  We will not pursue this here.
\end{itemize}
\end{rem}

\begin{proof}[{Proof of Theorem $\ref{thm}$}]
First assume \eqref{gp} for some $p\in[1,+\infty]$ and some $\eta>0$. By Proposition \ref{proposition}, we know that this estimate self-improves into \eqref{eq:am}.
Fix a point $z\in M$ and consider the function $f=p_t(\cdot,z)$. Then \eqref{eq:am} yields 
$$ | p_{2t}(x,z) - p_{2t}(y,z) | \lesssim \left( \frac{d(x,y)}{\sqrt{t}} \right)^{\eta}  \left|B_{\sqrt{t}}\right|^{-1/p} \|p_t(\cdot,z)\|_p, $$
uniformly for a.e. $x,y$ with $d(x,y)\leq \sqrt{t}$ and where $B_{\sqrt{t}}$ is any ball of radius $\sqrt{t}$ containing $x,y$ (in particular, $p_t$ is continuous and $p_t(x,x)$ has a meaning).
It follows  from \eqref{d} and \eqref{UE} that 
\begin{equation} \label{eq:hk-Lp-est}
\|p_t(\cdot,z)\|_p\lesssim \left[V(z,\sqrt{t})\right]^{\frac{1}{p}-1}
\end{equation}
and that 
$$V^{-1}(z,\sqrt{t})\lesssim p_{2t}(z,z).$$
For these two classical facts, see for instance \cite[Theorem 3.1]{coupre}.
Hence
\begin{align*}
 | p_{2t}(x,z) - p_{2t}(y,z) | & \lesssim \left( \frac{d(x,y)}{\sqrt{t}} \right)^{\eta}  \left(\frac{V(z,\sqrt{t})}{\left|B_{\sqrt{t}}\right|} \right)^ {1/p} V(z,\sqrt{t})^{-1} \\
  & \lesssim \left( \frac{d(x,y)}{\sqrt{t}} \right)^{\eta}  \left(\frac{V(z,\sqrt{t})}{\left|B_{\sqrt{t}}\right|} \right)^ {1/p} p_{2t}(z,z). 
\end{align*}

Note that this estimate is nothing but a slightly weaker form of  the classical H\"older estimate  \eqref{ho} from the introduction.

In particular, for $x=z$ and every $y\in B(x,\sqrt{t})$ we deduce that
\begin{align}
 | p_{2t}(x,x) - p_{2t}(y,x) | & \lesssim \left( \frac{d(x,y)}{\sqrt{t}} \right)^{\eta}  p_{2t}(x,x). \label{eq:fff}
\end{align}
It is well-known that \eqref{LE} follows (see for instance  \cite[Theorem 3.1]{coupre}).

Assume now \eqref{LE}. Since we have assumed   \eqref{UE},
it follows, through the equivalence of $\eqref{UE}+\eqref{LE}$ with the parabolic Harnack inequality (see  \cite[Proposition 3.2]{parma} or \cite[Theorems 2.31-2.32]{GSC}), that there exist $\theta\in(0,1)$ such that, for a.e. $x,y\in B_r$, $0\le r<\sqrt{t}$, and a.e. $z\in M$
\begin{equation}\label{harna}
 | p_{t}(x,z) - p_{t}(y,z) | \lesssim \frac{1}{\sqrt{V(z,\sqrt{t})V(x,\sqrt{t})}}\left( \frac{r}{\sqrt{t}} \right)^{\theta}\end{equation}
 (this is yet another version of \eqref{ho}).
On the other hand, \eqref{UE} and doubling imply that
\begin{equation}\label{gauss}
| p_{t}(x,z) - p_{t}(y,z) | \lesssim p_{t}(x,z) + p_{t}(y,z)  \lesssim \frac{1}{V(x,\sqrt{t})}\exp\left(-c\frac{d^2(x,z)}{\sqrt{t}}\right).
\end{equation}
If $d(x,z)\le \sqrt{t}$, $V(z,\sqrt{t})\simeq V(x,\sqrt{t})\simeq \left|B_{\sqrt{t}}\right|$ and  \eqref{harna} yields
$$| p_{t}(x,z) - p_{t}(y,z) | \lesssim \frac{1}{\left|B_{\sqrt{t}}\right|}\left( \frac{r}{\sqrt{t}} \right)^{\theta}.
$$
 If $d(x,z)\ge \sqrt{t}$, one multiplies the square roots of  \eqref{harna} and  \eqref{gauss} to obtain
 \begin{eqnarray*}
 | p_{t}(x,z) - p_{t}(y,z) | &\lesssim& \frac{1}{V^{1/4}(z,\sqrt{t})V^{3/4}(x,\sqrt{t})}\left( \frac{r}{\sqrt{t}} \right)^{\theta/2}\exp\left(-c\frac{d^2(x,z)}{\sqrt{t}}\right)\\
  & =&\frac{1}{V(x,\sqrt{t})}\left( \frac{r}{\sqrt{t}} \right)^{\theta/2}\left(\frac{V(x,\sqrt{t})}{V(z,\sqrt{t})}\right)^{1/4}\exp\left(-c\frac{d^2(x,z)}{\sqrt{t}}\right)\\
  & \lesssim &\frac{1}{\left|B_{\sqrt{t}}\right|}\left( \frac{r}{\sqrt{t}} \right)^{\theta/2},
 \end{eqnarray*}
where the last inequality uses again doubling.

Now we proceed as in \cite[Theorem 3.1]{coupre}.
We have just shown that
$$\|p_t(x,.)-p_t(y,.)\|_\infty\lesssim \frac{1}{\left|B_{\sqrt{t}}\right|}\left( \frac{r}{\sqrt{t}} \right)^{\theta/2}.$$
The heat semigroup being submarkovian,
$$\|p_t(x,.)-p_t(y,.)\|_1\leq 2.$$
It follows by H\"older inequality that for $1\le p<+\infty$
\begin{equation}\label{h}
\|p_t(x,.)-p_t(y,.)\|_{p'}\lesssim \left|B_{\sqrt{t}}\right|^{-1/p}\left( \frac{r}{\sqrt{t}} \right)^{\theta/2p},
\end{equation}
for a.e. $x,y\in B_r$, $0\le r<\sqrt{t}$.
Now 
\begin{eqnarray*}
|e^{-t\mathcal{L}}f(x)-e^{-t\mathcal{L}}f(y)|&\le& \int_M|p_t(x,z)-p_t(y,z)||f(z)|\,d\mu(z)\\
&\le &\|p_t(x,.)-p_t(y,.)\|_{p'}\|f\|_p,
\end{eqnarray*}
which together with \eqref{h} yields  \eqref{eq:am} with $\eta=\theta/2p$, hence  $(H_{p,p}^{\eta})$ by Proposition  \ref{proposition}.

\end{proof}

\begin{rem} \label{rem:sub} Proposition $\ref{proposition}$ and Theorem $\ref{thm}$ still hold in the context of sub-Gaussian estimates. Instead of \eqref{UE}, let us assume that the heat kernel satisfies for some $m>2$
\begin{equation}\tag{$U\!E_m$}
p_{t}(x,y)\lesssim
\frac{1}{V(x,t^{1/m})}\exp
\left(-\left(\frac{d(x,y)^m}{Ct}\right)^{1/(m-1)} \right),\ \forall~t>0,\,\mbox{a.e. }x,y\in
 M.\label{UEm} 
\end{equation} 
Then one can easily check that the above remains true by replacing everywhere the scaling factor $\sqrt{t}$ by $t^{1/m}$.
One could also consider the more general  heat kernel estimates from   {\rm \cite[Section 5]{HS}}, where the  equivalence with  matching  Harnack inequalities is proved, see also {\rm \cite{BGK}}. 
\end{rem}

 An alternative proof of  the second statement in Theorem $\ref{thm}$ can be given using {\rm \cite[Theorem 6]{Boutayeb}} instead of Proposition $\ref{proposition}$. We leave the details to the reader.
 Conversely, a natural follow-up of the end of the proof  of Theorem $\ref{thm}$ is to get the results from   \cite{Boutayeb}, that is the extension of \cite[Theorem 4.1]{coupre} to the doubling setting. There is nothing essentially new here, but we shall give a proof  of  \cite[Proposition 10]{Boutayeb} for the sake of completeness.  
 
We first need  to introduce the notion of reverse doubling. It is known (see \cite[Proposition 5.2]{GH}),  that, if  $M$ is unbounded, connected, and satisfies \eqref{dnu}, one has a  so-called reverse doubling volume property, namely there
exist
 $0<\nu'\leq \nu$ and $c>0$ such that,
for all $r\geq s>0$ and $x\in M$
\begin{equation*}
c\left(\frac{r}{s}\right)^{\nu'}\leq \frac{V(x,r)}{V(x,s)}.
\end{equation*}
Let us say that $(M,d,\mu)$ satisfies $(V\!D_{\nu,\nu'})$ if, for all $r\geq s>0$ and $x\in M$,
\begin{equation*}
c\left(\frac{r}{s}\right)^{\nu'}\leq \frac{V(x,r)}{V(x,s)}\le C\left(\frac{r}{s}\right)^{\nu}.
\end{equation*}

For the sake of simplicity we shall  set ourselves in the Gaussian case ($m=2$ in the notation of Remark  \ref{rem:sub}), but the general case is similar.

 \begin{theorem}\label{boubis} Let $(M,d,\mu,{\mathcal E})$  be a metric measure Dirichlet space satisfying $(V\!D_{\nu,\nu'})$ and the upper Gaussian estimate \eqref{UE}. Then \eqref{LE}  holds if and only if, for some (all) $p\in(1,+\infty)$, some  $\alpha>\frac{\nu}{p}$ and $\alpha'>\frac{\nu'}{p}$, 
$$
|f(x)-f(y)|\lesssim \frac{1}{ V^{1/p}(x,d(x,y))} \left(d^\alpha(x,y)\|\mathcal{L}^{\alpha/2}f\|_p+d^{\alpha'}(x,y)\|\mathcal{L}^{\alpha'/2}f\|_p\right), 
$$
$\forall\,f\in \mathcal{D}, \,x,y\in M$.
\end{theorem}
 
 \begin{proof} Assume \eqref{LE}. Let $1<p<+\infty$,  $\alpha,\alpha'>0$ to be chosen later, and $k\in\N$ such that
 $k>\max\left(\frac{\alpha}{2},\frac{\alpha'}{2}\right)$. Let $f\in \mathcal{D}$.
 Thanks to \eqref{UE}  and to the fact that by reverse doubling $V(x,r)\to+\infty$ as $r\to+\infty$, $e^{-t\mathcal{L}}f\to 0$ in $L^2(M,\mu)$, as $t\to+\infty$
 (see \cite[Section 3.1.2]{Chen2} for details). Since  $e^{-t\mathcal{L}}f$ is bounded  in $L^1(M,\mu)$, $e^{-t\mathcal{L}}f\to 0$ in $L^p(M,\mu)$ by duality and interpolation. Thus one can 
 write
 $$f= c(k)\int_0^{+\infty}t^{k-1}\mathcal{L}^k e^{-t\mathcal{L}}f\,dt,$$
 hence
  \begin{eqnarray*}
|f(x)-f(y)|&\le& c(k)\int_0^{+\infty}t^{k-1}|\mathcal{L}^k e^{-t\mathcal{L}}f(x)-\mathcal{L}^k e^{-t\mathcal{L}}f(y)|\,dt\\
&=& c(k)\int_0^{+\infty}t^{k-1}|e^{-(t/2)\mathcal{L}}\mathcal{L}^k e^{-(t/2)\mathcal{L}}f(x)-e^{-(t/2)\mathcal{L}}\mathcal{L}^k e^{-(t/2)\mathcal{L}}f(y)|\,dt.
\end{eqnarray*}
Now for $r=d(x,y)$ and $0<\sqrt{t} \leq r$, we get from \eqref{eq:hk-Lp-est} and \eqref{dnu}
\begin{align*}
	& |e^{-(t/2)\mathcal{L}}\mathcal{L}^k e^{-(t/2)\mathcal{L}}f(x)-e^{-(t/2)\mathcal{L}}\mathcal{L}^k e^{-(t/2)\mathcal{L}}f(y)| \\
& \qquad \le \|p_t(x,.)-p_t(y,.)\|_{p'}\|\mathcal{L}^k e^{-(t/2)\mathcal{L}}f\|_p \\
& \qquad \le \left(\|p_t(x,.)\|_{p'}+\|p_t(y,.)\|_{p'}\right)\|\mathcal{L}^k e^{-(t/2)\mathcal{L}}f\|_p\\
& \qquad \lesssim V(x,\sqrt{t})^{-1/p} t^{-(k-\frac{\alpha}{2})}\|\mathcal{L}^{\alpha/2}f\|_p \\
& \qquad \lesssim V(x,r)^{-1/p} \left(\frac{r}{\sqrt{t}}\right)^{\nu/p} t^{-(k-\frac{\alpha}{2})}\|\mathcal{L}^{\alpha/2}f\|_p,
\end{align*}
where the last  inequality uses the analyticity of $(e^{-t\mathcal{L}})_{t>0}$ on $L^p(M,\mu)$.
For $0 \leq r <\sqrt{t}$, we can write as in the end of the proof of Theorem $\ref{thm}$,
  \begin{align*}& |e^{-(t/2)\mathcal{L}}\mathcal{L}^k e^{-(t/2)\mathcal{L}}f(x)-e^{-(t/2)\mathcal{L}}\mathcal{L}^k e^{-(t/2)\mathcal{L}}f(y)|\\&\le \|p_t(x,.)-p_t(y,.)\|_{p'}\|\mathcal{L}^k e^{-(t/2)\mathcal{L}}f\|_p
\lesssim \left|B_{\sqrt{t}}\right|^{-1/p}\left( \frac{r}{\sqrt{t}} \right)^{\theta/2p}t^{-(k-\frac{\alpha'}{2})}\|\mathcal{L}^{\alpha'/2}f\|_p. \end{align*}
Now reverse doubling yields
 \begin{align*} &|e^{-(t/2)\mathcal{L}}\mathcal{L}^k e^{-(t/2)\mathcal{L}}f(x)-e^{-(t/2)\mathcal{L}}\mathcal{L}^k e^{-(t/2)\mathcal{L}}f(y)| \\
 \lesssim &V(x,r)^{-1/p} \left( \frac{r}{\sqrt{t}} \right)^{\frac{\theta}{2p}+\frac{\nu'}{p}}t^{-(k-\frac{\alpha'}{2})}\|\mathcal{L}^{\alpha'/2}f\|_p.
  \end{align*}
Finally
 \begin{align*}
|f(x)-f(y)|
& \lesssim V(x,r)^{-1/p} \|\mathcal{L}^{\alpha/2}f\|_p \, r^{\nu/p} \int_0^{r^2} t^{k-1} t^{-\frac{\nu}{2p}-k+\frac{\alpha}{2}} \,dt \\
& \quad + V(x,r)^{-1/p} \|\mathcal{L}^{\alpha'/2}f\|_p \, r^{\frac{\theta}{2p}+\frac{\nu'}{p}} \int_{r^2}^{+\infty} t^{k-1} t^{-\frac{\theta}{4p}-\frac{\nu'}{2p}-k+\frac{\alpha'}{2}} \,dt.\end{align*}
The  above integrals converge if $\alpha>\frac{\nu}{p}$ and $\alpha'<\frac{\theta}{2p}+\frac{\nu'}{p}$, in which case one obtains
$$|f(x)-f(y)|
 \lesssim V(x,r)^{-1/p}\left( r^\alpha  \|\mathcal{L}^{\alpha/2}f\|_p+r^{\alpha'}  \|\mathcal{L}^{\alpha'/2}f\|_p\right).
 $$
One can choose any $\alpha>\frac{\nu}{p}$ and some $\alpha'>\frac{\nu'}{p}$.
The converse is easy, see \cite[Theorem 6]{Boutayeb}.
\end{proof}

\begin{rem} One can take $\alpha=\alpha'$ if $\nu=\nu'$, recovering in particular the polynomial volume growth case from {\rm \cite[Theorem 4.1]{coupre}}.
\end{rem}

In  \cite[Thm. 5.2]{coupre}, it is proved that if the volume growth is polynomial of exponent $\nu\ge 2$, then $(R_p)$ and $(P_{p})$ for $\nu<p<+\infty$ imply \eqref{LE}.
Using the equivalence between  $(G_p)$ and \eqref{mult}, it is easy to see that the same proof works with $(G_p)$ instead of $(R_p)$. Our next theorem   extends this result to the doubling case, and in addition its proof is more direct. 

\begin{theorem}\label{nantes} Let $(M,d,\mu, {\mathcal E})$ be a metric  measure   Dirichlet space with a ``carr\'e du champ'' satisfying \eqref{dnu}.
Assume $(G_p)$ and $(P_{p})$ for some $p\in(\nu,+\infty)$. Then  \eqref{LE} holds.
\end{theorem}

\begin{proof} Replacing  $f$ with $e^{-t\mathcal{L}} f$ in $(P_p)$, we have, for every $t>0$ and every ball $B_r$ of radius $r>0$,$$ p\text{-}\osc_{B_r}(e^{-t\mathcal{L}}f)\lesssim  r  \left(\aver{B_r} | \nabla e^{-t\mathcal{L}} f |^p d\mu \right)^{1/p}.$$
If $B_{\sqrt{t}}$ is concentric with $B_r$ and $\sqrt{t}\ge r$,
$$ \left(\aver{B_r} | \nabla e^{-t\mathcal{L}} f |^p d\mu \right)^{1/p}  \lesssim \left( \frac{\left|B_{\sqrt{t}}\right|}{|B_r|}\right)^ {1/p} \left(\aver{B_{\sqrt{t}}} | \nabla e^{-t\mathcal{L}} f |^p d\mu \right)^{1/p} ,$$
hence by \eqref{dnu}
\begin{align*}
\left(\aver{B_r} | \nabla e^{-t\mathcal{L}} f |^p d\mu \right)^{1/p}  & \lesssim \left( \frac{\sqrt{t}}{r}\right)^ {\nu\over p}  \left|B_{\sqrt{t}}\right|^{-1/p}  \||\nabla e^{-t\mathcal{L}} f|\|_p\\& \lesssim \left( \frac{\sqrt{t}}{r}\right)^ {\nu\over p}  \left|B_{\sqrt{t}}\right|^{-1/p} \frac{ \| f\|_p}{\sqrt{t}},
\end{align*}
where the last inequality follows from  $(G_p)$. Gathering the two above estimates yields
$$p\text{-}\osc_{B_r}(e^{-t\mathcal{L}}f)\lesssim \left( \frac{r}{\sqrt{t}}\right)^{1- {\nu\over p}} \left|B_{\sqrt{t}}\right|^{-1/p} \| f\|_p,$$
that  is \eqref{gp} with $\eta=1-\frac{\nu}{p}\in(0,1)$ since  $p>\nu$. By Proposition \ref{prop:ppgp}, \eqref{UE} also holds. We conclude by applying Theorem \ref{thm}.
\end{proof}

\begin{rem} In the above statement, $(P_p)$ is necessary, since \eqref{LE} implies $(P_2)$, but $(G_p)$ is not, as  the example of conical manifolds shows (see {\rm \cite{CHQL}}).
\end{rem}

\begin{rem}\label{Ginfty} For $p=+\infty$, the above proof with the obvious modifications shows that $(G_\infty)$ together with  \eqref{UE} implies \eqref{LE}. This also follows from  {\rm \cite[Corollary 2.2]{CS2}}
and the fact that $(G_\infty)$ and \eqref{UE} imply the matching pointwise estimate of the gradient of the heat kernel.
\end{rem}

Let us give an alternative proof of Theorem \ref{nantes}, which is more direct, but does not shed the same light on the range $2\le p\le\nu$ (see Section \ref{section:2}
below) as the above one.

We shall start with a lemma which is close to \cite[Theorem 5.1]{HaKo} and to several statements  in  \cite{CLi} (for the polynomial volume growth case),
but we find it useful to formulate and prove it in the following simple and natural way, which is in fact inspired  by 
 \cite[Theorem 3.2]{HaKo}.  
 
 \begin{lemma}\label{lemma:morrey} Let $(M,d,\mu, {\mathcal E})$ be a metric  measure   Dirichlet space with a ``carr\'e du champ'' satisfying \eqref{dnu}.  Then $(P_p)$ for some $p>\nu$ implies the following Morrey inequality: for every function $f\in \mathcal{D}$ and almost every $x,y\in M$,
\begin{equation}\label{momo}
|f(x)-f(y)|\lesssim \frac{d(x,y)}{ V^{1/p}(x,d(x,y))} \||\nabla f|\|_p.
\end{equation}
\end{lemma}

\begin{proof} Let $x,y$ be  Lebesgue points for $f$. Let $B_i(x)=B(x,2^{-i}d(x,y))$, for $i\in\N_0$.
As in Lemma \ref{lemma:meyers}, one has
$$
\left|f(x) - \aver{B_0(x)} f d\mu \right|\le\sum_{i\geq 0}p\text{-}\osc_{B_i(x)} (f),
$$
Then using  $(P_p)$ and \eqref{dnu} which yields $|B_0(x)|\lesssim 2^{i\nu} |B_i(x)|$, we can write
\begin{eqnarray*}
\left|f(x) - \aver{B_0(x)} f d\mu \right| &\lesssim & \sum_{i\geq 0} 2^{-i} d(x,y) \left(\aver{B_{i}(x)}\left|\nabla f\right|^p\,d\mu\right)^{1/p} \\
&\lesssim & \sum_{i\geq 0} 2^{-i} d(x,y) \left(\frac{2^{i\nu}}{ |B_0(x)|}\right)^{1/p} \left(\int_{B_{i}(x)}\left|\nabla f\right|^p\,d\mu\right)^{1/p} \\
&\lesssim & \left(\sum_{i\geq 0} 2^{-i(1-\frac{\nu}{p})}\right) d(x,y)|B_0(x)|^{-1/p} \left\| |\nabla f| \right\|_p\\
&\lesssim & d(x,y)|B_0(x)|^{-1/p} \left\| |\nabla f| \right\|_p,
\end{eqnarray*}
where we used $p>\nu$.
Similarly we have
\begin{eqnarray*}
\left|f(y) - \aver{B_0(y)} f d\mu \right| & \lesssim d(x,y)|B_0(y)|^{-1/p} \left\| |\nabla f| \right\|_p \\
& \lesssim d(x,y)|B_0(x)|^{-1/p} \left\| |\nabla f| \right\|_p,
\end{eqnarray*}
where $|B_0(x)| \simeq |B_0(y)|$ follows from doubling.
Finally,  as in Lemma \ref{lemma:meyers}, 
\begin{equation*}
\left|\aver{B_0(x)} f d\mu - \aver{B_0(y)} f d\mu \right|\leq \left(\aver{2B_0(x)} \left|f - \aver{2B_0(x)} f d\mu \right|^p\,d\mu\right)^{1/p},
\end{equation*}
and by 
 $(P_p)$
 \begin{equation*}
\left|\aver{B_0(x)} f d\mu - \aver{B_0(y)} f d\mu \right|\lesssim  d(x,y)|B_0(x)|^{-1/p} \left\| |\nabla f| \right\|_p.
\end{equation*}
The claim follows.
\end{proof}

 We can now derive  Theorem \ref{nantes} easily. 
 Replacing $f$  with $e^{-t\mathcal{L}}f$ in the conclusion of Lemma \ref{lemma:morrey}   and applying $(G_p)$ yields
\begin{equation*}
|e^{-t\mathcal{L}}f(x)-e^{-t\mathcal{L}}f(y)| \lesssim  \frac{d(x,y)} {V^ {1/p}(x,d(x,y))}  \| |\nabla e^{-t\mathcal{L}}f|\|_p \lesssim \frac{d(x,y)}{\sqrt{t}} \, \left[V(x,d(x,y))\right]^ {-1/p}   \|f\|_p,
\end{equation*}
hence by \eqref{dnu}
\begin{equation}\label{dis}
|e^{-t\mathcal{L}}f(x)-e^{-t\mathcal{L}}f(y)|\lesssim \left( \frac{d(x,y)}{\sqrt{t}} \right)^{1-\frac{\nu}{p}}|B_{\sqrt{t}}|^ {-1/p}  \|f\|_p,
\end{equation}
for every $f\in \mathcal{D}$ and every  ball $B_{\sqrt{t}}$ with  radius $\sqrt{t}\ge d(x,y)$ and containing $x$. 
Let now $B_r$ be concentric to $B_{\sqrt{t}}$ with  radius $\sqrt{t}$ such that $0<r\le \sqrt{t}$. Since  $p>\nu$, it follows from \eqref{dis} that
\begin{equation*}
\esssup_{x,y\in B_r}|e^{-t\mathcal{L}}f(x)-e^{-t\mathcal{L}}f(y)|\lesssim \left( \frac{r}{\sqrt{t}} \right)^{1-\frac{\nu}{p}}|B_{\sqrt{t}}|^ {-1/p}  \|f\|_p.
\end{equation*}
This is nothing but \eqref{eq:am}  with $\eta=1-\frac{\nu}{p}\in(0,1)$, and we conclude by using \eqref{UE} as in the beginning of the proof of Theorem \ref{thm}.

\bigskip

\section{Poincar\'e inequalities and heat kernel bounds: the $L^2$ theory} \label{section:2}

The so-called De Giorgi property  or Dirichlet property on the growth of the Dirichlet integral for  harmonic functions was introduced by De Giorgi in \cite{DeG}, for $\mathcal{L}$ a  second order divergence form differential operator with real coefficients on $\R^n$: there exists $\epsilon\in(0,1)$ such that for all $r\leq R$, every pair of concentric balls $B_r,B_R$ with radii $r,R$ and all functions $u\in W^{1,2}(\R^n)$ harmonic in $2B_R$, i.e. $\mathcal{L} u=0$ in $2B_R$, one has
\begin{equation} \left(\aver{B_r} |\nabla u|^2 d\mu \right)^{1/2} \lesssim \left(\frac{R}{r}\right)^\epsilon \left(\aver{B_R} |\nabla u|^2 d\mu\right)^{1/2}. \label{eq:dg} \end{equation}
The De Giorgi property was subsequently used in many works  and in various situations to prove H\"older regularity for  solutions of inhomogeneous elliptic equations and systems
(see for instance \cite{Gia}).

The idea  to look at the  heat equation as a Laplace equation where the RHS is a time derivative, and to deduce parabolic regularity results  from elliptic ones by using a non-homogeneous equivalent version of De Giorgi property was introduced in \cite{Aus} for  $\mathcal{L}$ a  second order operator
 in divergence form on $\R^n$.   In \cite{AC2}, the same ideas are applied in  a discrete geometric setting, and the role of Poincar\'e inequalities clearly appears to ensure the elliptic regularity and  the equivalence between the homogeneous  and non-homogeneous  versions of De Giorgi. This is the approach we will follow here,  while  taking  full advantage of Theorem \ref{thm}. We shall consider the following  non-homogeneous version of  De Giorgi property.  With the help of Lemma \ref{lemma-end} below, one shows that this formulation is a priori weaker than the one in  \cite[Proposition 4.4]{AC2}.  We shall see in the proof of Proposition  \ref{propDG} in Appendix \ref{AppDG} that under  \eqref{UE} it is equivalent to \eqref{eq:dg}.
 
 \begin{definition}[De Giorgi property] Let $(M,d,\mu, {\mathcal E})$ be a metric  measure   Dirichlet space with a ``carr\'e du champ'' and $\mathcal{L}$ the associated operator. We say that \eqref{DG2} holds if the following is satisfied: for all $r\leq R$, every pair of concentric balls $B_r,B_R$ with respective radii $r$ and $R$, and for every function $f\in {\mathcal D}$, one has
\begin{equation}
\left(\aver{B_r} |\nabla f|^2 d\mu \right)^{1/2} \lesssim \left(\frac{R}{r}\right)^\epsilon \left[\left(\aver{B_R} |\nabla f|^2 d\mu\right)^{1/2} + R \| \mathcal{L} f \|_{L^\infty(B_R)}\right].
\tag{$DG_{2,\epsilon}$} \label{DG2}
\end{equation}
We sometimes omit the parameter $\epsilon$, and write $(DG_2)$ if \eqref{DG2} is satisfied for some $\epsilon \in (0,1)$.
\end{definition}


Let us now state the counterpart of a result  of \cite{AC2} in the discrete setting. For the convenience of the reader, we give a proof  in Appendix \ref{AppDG}.

\begin{proposition} \label{propDG}
Let $(M,d,\mu, {\mathcal E})$  be  a doubling metric  measure   Dirichlet space with a ``carr\'e du champ''. Then $(P_2)$ implies $(DG_2)$. 
\end{proposition}

We are now in a position to give a simple proof of the main result of \cite{Gr0}, \cite{S}, and \cite{ST3}. For simplicity let us denote in what follows, for $B$ a ball and $f\in L^2_{loc}(M,\mu)$:
$$ \osc_B(f) :=2\text{-}\osc_B(f)= \left(\aver{B} |f-\aver{B} f \,d\mu|^2 \,d\mu\right)^{1/2}.$$

\begin{theorem} \label{rem:ep}
Let $(M,d,\mu, {\mathcal E})$  be  a doubling metric  measure   Dirichlet space with a ``carr\'e du champ''. Then $(P_2)$ implies \eqref{LE}. 
\end{theorem}

\begin{proof} Applying  $(P_2)$  to $e^{-t\mathcal{L}}f$ for $t>0$ and $f\in L^{2}(M,\mu)$ on a ball $B_r$ for $r>0$   yields
\begin{equation}
\osc_{B_r}(e^{-t\mathcal{L}}f)
 \lesssim  r\left(\aver{B_r} | \nabla e^{-t\mathcal{L}} f |^2 d\mu \right)^{1/2}.\label{osnab}
 \end{equation}
According to Proposition \ref{propDG}, \eqref{DG2} holds for some $\eps>0$, hence
\begin{eqnarray}\label{nabnab}
&&\left(\aver{B_r} | \nabla e^{-t\mathcal{L}} f |^2 d\mu \right)^{1/2}\\&\lesssim&\left( \frac{\sqrt{t}}{r} \right)^ {\epsilon} \left[\left(\aver{B_{\sqrt{t}}} |\nabla e^{-t\mathcal{L}} f|^2 d\mu\right)^{1/2} + \sqrt{t} \esssup_{x\in B_{\sqrt{t}}} |\mathcal{L} e^{-t\mathcal{L}} f(x)|\right]\nonumber\\&\leq&\left( \frac{\sqrt{t}}{r} \right)^ {\epsilon}  \left(|B_{\sqrt{t}}|^{-1/2} \| |\nabla e^{-t\mathcal{L}} f| \|_2 +\sqrt{t} \esssup_{x\in B_{\sqrt{t}}} |\mathcal{L} e^{-t\mathcal{L}} f(x)|\right)\nonumber
\end{eqnarray}
for some $\epsilon\in(0,1)$,  $0<r\leq \sqrt{t}$ and $B_{\sqrt{t}}$ with radius $\sqrt{t}$  concentric to $B_r$.
By $(G_2)$, 
\begin{equation}\label{nabsq}
\| |\nabla e^{-t\mathcal{L}} f| \|_2\lesssim \frac{\|f\|_2}{\sqrt{t}}.
\end{equation}
Now recall that under our assumptions, \eqref{UE} holds thanks to Proposition \ref{prop:ppgp}. By \cite[Corollary 3.3]{Gr1} (one can also use the complex time bounds of \cite[Proposition 4.1]{CCO} and a Cauchy formula) the kernel of the operator $t\mathcal{L} e^{-t\mathcal{L}}$ also satisfies pointwise Gaussian estimates.
It follow that
\begin{equation}\label{nabinf}
\esssup_{x\in 2B_{\sqrt{t}}} |t\mathcal{L} e^{-t\mathcal{L}} f(x)|\lesssim |B_{\sqrt{t}}|^{-1/2} \|f\|_2.
\end{equation}
Putting together \eqref{osnab}, \eqref{nabnab}, \eqref{nabsq} and \eqref{nabinf} yields
\begin{align*}
\osc_{B_r}(e^{-t\mathcal{L}}f) \lesssim \left( \frac{r}{\sqrt{t}} \right)^ {1-\epsilon} |B_{\sqrt{t}}|^{-1/2} \|f\|_{2},
\end{align*}
that is, $(H_{2,2}^\eta)$ with $\eta=1-\epsilon>0$.  This implies \eqref{LE} according to Theorem \ref{thm}.
\end{proof}

The  original proofs of Theorem \ref{rem:ep} went through the parabolic Harnack inequality; some \cite{S, SU, parma, SA, ST3} used   a  parabolic Moser iteration, another one \cite{Gr0} tricky geometric arguments.  In \cite[Section 4.2]{HS},  a shorter proof was given, which went in three steps (with a fourth one, borrowed from \cite{FS},  to deduce parabolic Harnack from \eqref{LE}). The first one is to derive an elliptic regularity estimate from \eqref{d} and $(P_2)$.  We do not change this step, which relies on the elliptic Moser iteration; we give a proof for the sake of completeness in Proposition \ref{propER} below.  The second step is to obtain \eqref{UE}. Our approach in Proposition \ref{prop:ppgp} is particularly simple since $p=2$. The third step is a lower bound on the Dirichlet heat kernel inside a ball whose radius is the square root of the time under consideration. This is not trivial (see  \cite[pp. 1457-1462]{HS}) and here lies our main simplification.  We first push step one a little further by deducing $(DG_2)$ from the elliptic regularity.  We could then deduce the parabolic regularity as in \cite[Section 4]{AC2}). Instead, we use  the self-improvement of H\"older regularity estimates on the semigroup from Proposition \ref{proposition} and Theorem \ref{thm}.

\medskip
Introduce the
scale-invariant local Sobolev inequality
\begin{equation*}\label{Sq}
\tag{$LS_q$}
\|f \|_{q}^2\lesssim \frac{1}{V^{1-\frac{2}{q}}(x,r)}\left(\|f\|^2_2+r^2\mathcal{E}(f)\right),
\end{equation*}
for every ball $B=B(x,r)$,  every $f\in\mathcal{F}$ supported in $B(x,r)$, and for some  $q>2$.
This inequality was introduced in \cite{S} and was shown, under \eqref{d}, to be equivalent  to \eqref{due} in the Riemannian setting. The equivalence was stated in  our more general setting in \cite{ST2}. See also \cite{BCS} for many reformulations of  \eqref{Sq}, an alternative proof of the equivalence with \eqref{due}, and more references.


The main aim of \cite{HS} is to prove that the elliptic Harnack inequality, or an equivalent elliptic regularity estimate, together with \eqref{Sq}, or equivalently \eqref{UE}, implies the parabolic Harnack inequality.   It is enough in this respect to prove \eqref{LE}, since as we already said the parabolic Harnack inequality follows from $\eqref{UE}+\eqref{LE}$. This phenomenon falls in the circle of the ideas we are developing in the present work, and, using  a transition trick from estimates for harmonic functions to estimates for all functions together with  Theorem \ref{thm}, we will now offer a simple proof of \cite[Theorem 3.1]{HS}.  Let us say that  $u \in \mathcal{F}$ is harmonic on a ball $B$ if $L u=0$ in the weak sense on $B$. Note that the following statement involves  $\diam(M)$ as we want to treat by the same token the cases $M$ bounded and unbounded. In a first reading one can certainly assume $\diam(M)=+\infty$.

\begin{theorem} \label{thm:2} Let $(M,d,\mu, {\mathcal E})$  be  a doubling metric  measure   Dirichlet space with a ``carr\'e du champ'' satisfying \eqref{Sq} for some $q>2$. 
Assume that the following elliptic regularity estimate holds: there exists $\alpha>0$ and $\delta\in(0,1)$ such that for every $x_0 \in M$, $R>0$ with $R< \delta\, \diam(M)$, $u \in \mathcal{F}$ harmonic in $B(x_0,R)$ and $x,y \in B(x_0,R/2)$, one has
\begin{equation} \label{eq:ER} \tag{$ER$}
	|u(x)-u(y)|
		\lesssim \left(\frac{d(x,y)}{R}\right)^\alpha
		\osc_{B(x_0,R)}(u).
\end{equation}
Then  \eqref{LE} follows. 
\end{theorem}

\begin{rem}\label{PFK} It is  known that  $(P_2)$ implies  \eqref{Sq} for some $q>2$, see for instance {\rm\cite[Theorem 2.1]{S}, \cite[Section 5]{CS3}}. We shall also see in Proposition ${\rm \ref{propER}}$ below that   $(P_2)$ implies  \eqref{eq:ER}. Thus Theorem $\ref{thm:2}$ gives back Theorem  $\ref{rem:ep}$.

\end{rem}

Before we start the proof of Theorem $\ref{thm:2}$, recall  that \eqref{Sq} for some $q>2$ implies the following relative Faber-Krahn inequality
\begin{equation} \tag{$FK$} \label{FK}
	\left(\int_\Omega |f|^2\,d\mu\right)^{1/2}
	\lesssim r\left(\frac{|\Omega|}{V(x,r)}\right)^\beta  \left(\int_\Omega |\nabla f|^2\,d\mu\right)^{1/2}
\end{equation}
for some $\beta>0$, all balls $B(x,r)$, $x\in M$, $r\in(0,\delta \,\diam(M))$ with some $\delta<1$, and all  $f\in\mathcal{F}$ supported in $\Omega\subset B(x,r)$. See for instance \cite[Theorem 2.5]{HS}, as well as \cite[Section 3.3]{BCS}.

In particular, one has
\begin{equation}  \label{FKR}
	\left(\int_{B(x,r)} |f|^2\,d\mu\right)^{1/2}
	\lesssim r \left(\int_{B(x,r)} |\nabla f|^2\,d\mu\right)^{1/2},
\end{equation}
for all balls $B(x,r)$, $x\in M$, $r\in(0,\delta \,\diam(M))$ with some $\delta<1$, and all  $f\in\mathcal{F}$ supported in $B(x,r)$.

We will need the following result inspired by  \cite[Lemma 4.2]{AC2}. Note that the role classically played by ellipticity in such Lax-Milgram type arguments is played here by \eqref{FKR}.

\begin{lemma} \label{lem:LM} Let $(M,d,\mu, {\mathcal E})$  be  a doubling metric  measure   Dirichlet space with a ``carr\'e du champ'' satisfying \eqref{FKR}. Let $f\in {\mathcal D}$ and consider an open ball $B \subset M$. Then, there exists $u\in {\mathcal F}$ such that $f-u\in {\mathcal F}$ is supported in the ball $B$ and $u$ is harmonic in $B$: for every $\phi \in {\mathcal F}$ supported on $B$ 
$$ \int_M d\Gamma(u,\phi)=0,$$
where we recall that $d\Gamma$ is the energy measure associated with the Dirichlet form ${\mathcal E}$.
Moreover, we have
\begin{equation} \left(\aver{B} |\nabla(f-u)|^2 \, d\mu \right)^{1/2} + \left(\aver{B} |\nabla u|^2 \, d\mu \right)^{1/2} \lesssim \left(\aver{B} |\nabla f|^2 \, d\mu \right)^{1/2}. \label{eq:grad} \end{equation}
\end{lemma}

\begin{proof}  Consider the space of functions 
$$ {\mathcal H} := \left\{ \phi \in {\mathcal F}\subset L^2, \ \textrm{supp}(\phi) \subset B \right\}.$$
Then, due to \eqref{FKR} the application 
$$\phi \mapsto \|\phi\|_{{\mathcal H}}:= \| |\nabla \phi |\|_{L^2(B)}$$
defines a norm on ${\mathcal H}$.  Consequently, ${\mathcal H}$ equipped with this norm is a Hilbert space, with the scalar product
$$ \langle \phi_1,\phi_2\rangle_{{\mathcal H}} := \int_{B} d\Gamma(\phi_1,\phi_2).$$ 
Since $f\in {\mathcal D} \subset {\mathcal F}$ then the linear form 
$$ \phi \mapsto \int_{B} d\Gamma(f,\phi)$$
is continuous on ${\mathcal H}$. By the   Riesz representation theorem, there exists $v\in{\mathcal H}$ such that for every $\phi \in  {\mathcal H}$ 
$$ \int_{B} d\Gamma(f,\phi) = \int_{B} d\Gamma(v,\phi).$$
We set $u:= f-v$ so that $v=f-u$ being in ${\mathcal H}$  is supported in  $B$. Moreover for every $\phi\in {\mathcal H}$, $\phi$ is supported in $B$ so the previous equality yields
$$ \int_{M} d\Gamma(u,\phi) =  \int_{M} d\Gamma(f,\phi) - \int_{M} d\Gamma(v,\phi) =0.$$
Observe that since $f-u$ is supported in $B$ and $u$ is harmonic on $B$
\begin{align*}
\aver{B} |\nabla(f-u)|^2 \, d\mu & = \frac{1}{|B|}\mathcal{E}(f-u,f-u)
=\frac{1}{|B|}\int_M \, d\Gamma(f-u,f-u)  \\
&=\frac{1}{|B|}\int_{M} \, d\Gamma(f,f-u). 
\end{align*}
Then by the locality property of the carr\'e du champ and since $f-u$ is supported in the ball $B$, we deduce that
\begin{align*}
\aver{B} |\nabla(f-u)|^2 \, d\mu & \leq \aver{B} |\nabla f| |\nabla(f-u)| \, d\mu  \\
 & \leq  \left(\aver{B} |\nabla (f-u)|^2 \, d\mu \right)^{1/2} \left(\aver{B} |\nabla f|^2 \, d\mu\right)^{1/2}.
\end{align*}
So it follows that 
$$ \left(\aver{B} |\nabla(f-u)|^2 \, d\mu \right)^{1/2} \lesssim \left(\aver{B} |\nabla f|^2 \, d\mu \right)^{1/2}$$
which allows us to prove \eqref{eq:grad}. 
\end{proof}

\begin{proof}[{Proof of Theorem $\ref{thm:2}$}] 


Let $u\in {\mathcal F}$ be a function harmonic on a ball $B_R=B(x_0,R)$ and write $B_{r}=B(x_0,r)$ for $r\leq R$ with $R \leq \delta\, \diam(M)$ (where we have chosen for $\delta$ the minimum of the two parameters in \eqref{FK} and \eqref{eq:ER}). From \eqref{eq:ER}, it follows that
\begin{equation} \osc_{B_r}(u) \lesssim \left(\frac{r}{R}\right)^{\alpha} \osc_{B_R}(u). \label{eq:ff} \end{equation}
Indeed, let $0<r\leq R/4$. According to  \eqref{eq:ER},  for every $x \in B_{r}$,
\begin{align*}
	&\left|u(x)-\aver{B_{r}} u(y) \,d\mu(y)\right|
	\leq \aver{B_{r}} |u(x)-u(y)|\,d\mu(y)\\
	& \quad \lesssim \left(\aver{B_{r}} \left(\frac{d(x,y)}{R}\right)^\alpha \,d\mu(y)\right)\osc_{B_R}(u)\\
	& \quad \lesssim \left(\frac{r}{R}\right)^\alpha \osc_{B_R}(u).
\end{align*}
Integrating over $B_{r}$ then gives
\begin{align*}
\osc_{B_{r}}(u)	 \lesssim \left(\frac{r}{R}\right)^\alpha \osc_{B_R}(u).
\end{align*}
The case $R/4\le r\le R$ is trivial.

We will now extend this estimate to non-harmonic functions, namely prove that
\begin{equation}\label{ernh}
 \osc_{B_r}(f)  \lesssim \left(\frac{r}{R}\right)^{\alpha} \osc_{B_{R}}(f) + \left(\frac{R}{r}\right)^{\nu/2} R^2  \left(\aver{B_R} |\mathcal{L}f|^2 \, d\mu\right)^{1/2}
\end{equation}
for all $f\in {\mathcal D}$ and  concentric balls $B_r,B_R$ with $0<r\le R$.
Let $f\in {\mathcal D}$. 
Since \eqref{FK} holds, one can invoke  Lemma \ref{lem:LM} below:  there exists $u\in {\mathcal F}$ harmonic on $B_R$ such that $f-u\in {\mathcal F}$ is supported in the ball $B_R$. 
One may write, using triangle inequality and \eqref{eq:ff},
\begin{align*} 
\osc_{B_r}(f) & \leq \osc_{B_r}(u) +  \osc_{B_r}(f-u) 
\lesssim \left(\frac{r}{R}\right)^{\alpha} \osc_{B_R}(u) +  \osc_{B_r}(f-u),
\end{align*}
hence by triangle inequality again
\begin{equation}\label{osc}
\osc_{B_r}(f)  \leq  \left(\frac{r}{R}\right)^{\alpha} \left(\osc_{B_R}(f) +  \osc_{B_R}(f-u)\right) + \osc_{B_r}(f-u).
\end{equation}

Let us start with estimating $\osc_{B_R}(f-u)$:
$$\osc_{B_R}(f-u)\lesssim \left( \aver{B_R} |f-u|^2 \, d\mu \right)^{1/2},$$
and since $f-u$ is supported in $B_R$, by \eqref{FK}  we have
\begin{equation*}
\left( \aver{B_R} |f-u|^2 \, d\mu \right)^{1/2}   \lesssim R \left( \aver{B_R} |\nabla (f-u)|^2 \, d\mu \right)^{1/2 }.
\end{equation*}
Now, since $f-u$ is supported on $B_R$ and $u$ is harmonic  on $B_R$
\begin{eqnarray*}
\aver{B_R} |\nabla(f-u)|^2 \, d\mu & = &\frac{1}{|B_R|}\mathcal{E}(f-u,f-u)
=\frac{1}{|B_R|}\int_M d\Gamma(f-u,f-u)  \\
&=&\frac{1}{|B_R|}\int_{M} d\Gamma(f,f-u) 
=\frac{1}{|B_R|}\int_{B_R}(f-u)\mathcal{L}f\,d\mu\\
&\le&\frac{1}{|B_R|}\int_{B_R}|f-u| |\mathcal{L}f|\,d\mu
\leq   \left(\aver{B_R} |f-u|^2\right)^{1/2} \left(\aver{B_R} |\mathcal{L}f|^2 \, d\mu\right)^{1/2}.
\end{eqnarray*}
From \eqref{FK}, it follows that
\begin{equation*}
\left(\aver{B_R} |\nabla(f-u)|^2 \, d\mu\right)^{1/2}\lesssim   R^{1/2} \left(\aver{B_R} |\nabla (f-u)|^2\right)^{1/4} \left(\aver{B_R} |\mathcal{L}f|^2 \, d\mu\right)^{1/4},
\end{equation*}
hence
\begin{equation}\label{sarko}
  \left(\aver{B_R} |\nabla(f-u)|^2 \, d\mu\right)^{1/2} \lesssim   R \left(\aver{B_R} |\mathcal{L}f|^2 \, d\mu\right)^{1/2}.
 \end{equation}

Gathering the above inequalities yields
\begin{equation}\label{edun}
\osc_{B_R}(f-u) \lesssim \left( \aver{B_R} |f-u|^2 \, d\mu \right)^{1/2}\lesssim R^2 \left( \aver{B_R} |\mathcal{L} f |^2 \, d\mu \right)^{1/2 }.
\end{equation}
Now for $\osc_{B_r}(f-u)$:
 by doubling 
\begin{align*}
\osc_{B_r}(f-u)  &\lesssim \left( \aver{B_r} |f-u|^2 \, d\mu \right)^{1/2}\\ &\lesssim \left(\frac{R}{r}\right)^{\nu/2}  \left( \aver{B_R} |f-u|^2 \, d\mu \right)^{1/2},
\end{align*}
therefore by \eqref{edun}
\begin{equation}\label{eded}
\osc_{B_r}(f-u)  \lesssim \left(\frac{R}{r}\right)^{\nu/2}R^2 \left(\aver{B_R} |\mathcal{L}f|^2 \, d\mu\right)^{1/2}.
\end{equation}

Finally, putting together \eqref{osc}, \eqref{edun}, \eqref{eded},
\begin{align*} \osc_{B_r}(f) & \lesssim \left(\frac{r}{R}\right)^{\alpha} \osc_{B_R}(f) + \left[ \left(\frac{r}{R}\right)^{\alpha}+\left(\frac{R}{r}\right)^{\nu/2}\right] R^2  \left(\aver{B_R} |\mathcal{L}f|^2 \, d\mu\right)^{1/2},
\end{align*}
which yields \eqref{ernh}.
A standard iteration argument, Lemma \ref{lemma-end} below, with 
$$ A(s):=s^{-1}\osc_{B_s}(f) \qquad \textrm{ and } \qquad B(s):=  \left(\aver{B_s} |\mathcal{L}f|^2 \, d\mu\right)^{1/2},$$
allows us to obtain for $\alpha'\in(0,\alpha)$
\begin{align*} \osc_{B_r}(f) & \lesssim \left(\frac{r}{R}\right)^{\alpha'} \left(\osc_{B_{R}}(f) + R^2  \|\mathcal{L}f\|_{L^\infty(B_{R})}\right).
\end{align*}
That holds for every $r\leq R$ with $R \leq \delta \,\diam(M)$.

Then if $M$ is unbounded, we may choose $R=\sqrt{t}$  and replace  $f$ with $e^{-t\mathcal{L}}f$,  which yields
\begin{align} \osc_{B_r}(e^{-t\mathcal{L}}f)  & \lesssim \left(\frac{r}{\sqrt{t}}\right)^{\alpha'} \left[\osc_{B_{\sqrt{t}}}(e^{-t\mathcal{L}}f)+ \|t\mathcal{L}e^{-t\mathcal{L}}f\|_{L^\infty(B_{\sqrt{t}})}\right] \nonumber \\
& \lesssim \left(\frac{r}{\sqrt{t}}\right)^{\alpha'}  |B_{\sqrt{t}}|^{-1/2} \|f\|_2, \label{eq:ffff}
\end{align}
where we used the Gaussian estimates for $t\mathcal{L}e^{-t\mathcal{L}}$ (see the proof of Theorem \ref{rem:ep}).  Property $(H^{\alpha'}_{2,2})$ follows and Theorem \ref{thm} yields \eqref{LE}. If the ambient space $M$ is bounded,  one can see that to get \eqref{LE} it is sufficient to check  \eqref{eq:ffff} for the scales $\sqrt{t} \lesssim \diam (M)$, which is exactly what we just have proved.
\end{proof}

It remains to prove the next lemma which follows ideas of \cite[Theorem 3.6]{Aus} and \cite[Lemma 2.1, Chapter III]{Gia0}.

\begin{lemma} \label{lemma-end} Let $0<r<R$ and consider a function $A:[r,R]\to {\mathbb R}^+$ such that
\begin{equation}\label{dtheta}
A(s)\lesssim \left(\frac{s'}{s}\right)^\theta A (s')
\end{equation}
for all  $s,s'$  such that  $r\le s\le s'\le R$ and for some $\theta>0$.
Let $B:[r,R]\to {\mathbb R}^+$, and 
assume that 
\begin{align} A(s)  & \lesssim  \left(\frac{s'}{s}\right)^{\epsilon} A(s') + \left(\frac{s'}{s}\right)^{\gamma} s' B(s'), \label{eq:dd} \end{align}
for every $r\leq s\leq s'\leq R$ and for some $\epsilon\in(0,1)$ and  $\gamma>0$.
Then
$$ A(r)  \lesssim  \left(\frac{R}{r}\right)^{\epsilon'} \left[A(R) + R \sup_{r\leq u\leq R} B(u)\right], $$
for every $\epsilon'\in(\epsilon,1)$. 
\end{lemma}

\begin{proof} Applying \eqref{eq:dd} with $s$ and $s'=Ks$ gives, for some numerical constant $C$, 
$$ A(s) \leq C K^\epsilon A(Ks)+ CK^{\gamma} Ks B(Ks) .$$
We choose $K> 1$ large enough such that $CK^\epsilon \leq K^{\epsilon'}$ for some fixed $\epsilon'\in(\epsilon,1)$. It follows that, for  $r\le s<Ks\leq R$, 
\begin{align*}
 A(s) & \leq  K^{\epsilon'} A(Ks)+  K^{\gamma+\epsilon'-\epsilon+1} s B(Ks)  \\
         & \leq  K^{\epsilon'} A(Ks)+  K^{\gamma+2} s \left(\sup_{r\leq u\leq R} B(u) \right).
\end{align*}

By iterating for $s=r, Kr, K^2r, ..., K^{\lfloor\lambda\rfloor-1}r$, where $\lambda$ is such that $K^\lambda r=R$, we deduce that
\begin{align*}
 A(r) \leq & K^{\lfloor\lambda\rfloor\eps'}A(K^{\lfloor\lambda\rfloor}r) + \left(\sum_{\ell=0}^{\lfloor\lambda\rfloor-1} (K^\ell r) K^{\ell  \epsilon'} \right)  K^{\gamma+2} \left(\sup_{r\leq u\leq R} B(u) \right)\\
= & K^{\lfloor\lambda\rfloor\eps'}A(K^{\lfloor\lambda\rfloor}r) + \left(\sum_{\ell=0}^{\lfloor\lambda\rfloor-1}  K^{\ell  (1+\epsilon')} \right)  r K^{\gamma+2} \left(\sup_{r\leq u\leq R} B(u) \right)\\
\lesssim & K^{\lfloor\lambda\rfloor\eps'}A(K^{\lfloor\lambda\rfloor}r) + K^{\lfloor\lambda\rfloor(1+\epsilon')}   r  \left(\sup_{r\leq u\leq R} B(u) \right)\\
\leq & K^{\lfloor\lambda\rfloor\eps'}\left[A(K^{\lfloor\lambda\rfloor}r) + K^{\lfloor\lambda\rfloor}   r  \left(\sup_{r\leq u\leq R} B(u) \right)\right]\\
\leq & K^{\lfloor\lambda\rfloor\eps'}\left[A(K^{\lfloor\lambda\rfloor}r) + R \left(\sup_{r\leq u\leq R} B(u) \right)\right].
 \end{align*}
The claim follows by using \eqref{dtheta}.
\end{proof}

\section{$L^p$ De Giorgi property and Caccioppoli inequality} \label{section:DG}

We pursue the same ideas as in Section \ref{section:2} but now for $p>2$. To this end,  we will rely on  $(G_p)$ and we will  introduce   $L^p$ versions, for $p\in[1,+\infty)$, of the De Giorgi property.

\begin{definition}[$L^p$ De Giorgi property] Let $(M,d,\mu, {\mathcal E})$ be a metric  measure   Dirichlet space with a ``carr\'e du champ'' and $\mathcal{L}$ the associated operator. For $p\in[1,+\infty)$ and $\epsilon\in(0,1)$, we say that \eqref{DG} holds if the following is satisfied: for all $r\leq R$, every pair of concentric balls $B_r,B_R$ with respective radii $r$ and $R$, and for every function $f\in {\mathcal D}$, one has
\begin{equation}
\left(\aver{B_r} |\nabla f|^p d\mu \right)^{1/p} \lesssim \left(\frac{R}{r}\right)^\epsilon \left[\left(\aver{B_R} |\nabla f|^p d\mu\right)^{1/p} + R \| \mathcal{L} f \|_{L^\infty(B_R)} \right].
\tag{$DG_{p,\epsilon}$} \label{DG}
\end{equation}
We sometimes omit the parameter $\epsilon$, and write $(DG_p)$ if \eqref{DG} is satisfied for some $\epsilon \in (0,1)$.
\end{definition}

\begin{rem}\label{lala} For  $f \in{\mathcal D}$ and $0<r<R$,
$$ \left(\aver{B_r} |\nabla f|^p d\mu \right)^{1/p} \lesssim \left(\frac{|B_R|}{|B_r|}\right)^{1/p} \left(\aver{B_R} |\nabla f|^p d\mu\right)^{1/p},$$
hence if  \eqref{dnu} holds, then
$$ \left(\aver{B_r} |\nabla f|^p d\mu \right)^{1/p} \lesssim \left(\frac{R}{r}\right)^{\nu/p} \left(\aver{B_R} |\nabla f|^p d\mu\right)^{1/p}.$$
Therefore if $p>\nu$, one always has $(DG_{p,\epsilon})$  with $\epsilon=\frac{\nu}{p}<1$. 
\end{rem}


We have just seen that, for $\nu<p<+\infty$, $(DG_{p})$ is a trivial consequence of \eqref{dnu}, and Proposition \ref{propDG} states that $(DG_{2})$ follows from \eqref{d} and $(P_2)$. In the range $2<p\le \nu$, property $(DG_{p})$ is more mysterious. The main result of this section is that, under $(P_2)$ and $(G_p)$ then $(DG_{q})$ holds for $2\le q<p$. 
In the $L^2$ situation, $(DG_2)$ is obtained from $(P_2)$ (as detailed in the Appendix \ref{AppDG}) through an extensive use of $L^2$-Caccioppoli inequality for subharmonic functions. Although this $L^2$ inequality is relatively easy, its $L^p$ version is unknown (we will prove it only for harmonic functions) and so it is not clear how  directly prove $(DG_p)$ for some $p>2$.  Beyond the fact that Proposition \ref{prop:Caccioppoli} will play a crucial role in Section \ref{sec:poi}, we feel that these questions are interesting.

\begin{proposition} \label{prop1} Let $(M,d,\mu, {\mathcal E})$  be  a doubling metric  measure   Dirichlet space with a ``carr\'e du champ''. Assume $(P_2)$ and  $(G_{p})$ for some $p\in(2,+\infty)$. Then  $({DG}_{q})$  holds for every $q\in[2,p]$.
 \end{proposition}

The proof of Proposition \ref{prop1} will make use of an  $L^p$ version of the Caccioppoli inequality. 
For $p=2$, integration by parts yields easily a $L^2$-Caccioppoli inequality (see for instance Lemma \ref{prop:Cacc-harmonic}) which enables one to deduce 
$(DG_{2})$  from  $(P_2)$ (Proposition \ref{propDG}).
But obtaining an $L^p$-Caccioppoli inequality for $p>2$ seems more difficult and cannot be handled using integration by parts directly. We use $(G_p)$ and the finite  propagation speed property instead, and obtain:

\begin{proposition}[$L^p$ Caccioppoli inequality] \label{prop:Caccioppoli} Let $(M,d,\mu, {\mathcal E})$  be  a doubling metric  measure   Dirichlet space with a ``carr\'e du champ''
satisfying  \eqref{UE}. Assume $(G_{p})$ for some $p\in[2,+\infty]$. Then for every $q\in(1,p]$,
 \begin{equation}\label{cacp}
 r\left(\aver{B_r} | \nabla f|^q d\mu \right)^{1/q} \lesssim \left(\aver{2B_r} |f|^q\, d\mu \right)^{1/q}+ r^2\left(\aver{2B_r} | \mathcal{L}f|^q \, d\mu \right)^{1/q}\end{equation}
for all $f\in {\mathcal D}$ and all balls  $B_r$ of radius $r$.
\end{proposition}

Note that it would be easier, and sufficient for our purposes in this section, to prove only a non-local version of the above inequality:
$$ r\left(\aver{B} | \nabla f|^q d\mu \right)^{1/q} \lesssim \sum_{\ell\geq 0}  2^{-\ell N} \left[\left(\aver{2^\ell B} |f|^q d\mu \right)^{1/q}+ r^2\left(\aver{2^\ell B} | \mathcal{L}f|^q d\mu \right)^{1/q}\right].$$
We feel anyway that \eqref{cacp} may be of independent interest, and that it is worth the extra effort, namely the use of the finite  propagation speed property.

Since the heat semigroup satisfies Davies-Gaffney estimates \eqref{eq:DG}, it is known (see e.g. \cite{ST2}, \cite{Sikora} and \cite[Section 3]{CS}) that $\sqrt{\mathcal{L}}$ satisfies the finite speed propagation property (with a speed equal to $1$ due to the normalization in \eqref{defd}) for solutions of the corresponding wave equation. Consequently for every even function $\varphi \in \mathcal{S}(\R)$ with $\supp \hat{\varphi} \subseteq [-1,1]$, every pair of Borel sets $E,F \subset M$ and every $r>0$, one has $\Eins_E \varphi(r \sqrt{\mathcal{L}}) \Eins_F =0$ if $\dist(E,F)>r$. This follows from the Fourier inversion formula and the bounded Borel functional calculus of $\sqrt{\mathcal{L}}$, cf. \cite[Lemma 4.4]{AMcM}.
Moreover, since the Dirichlet form is strongly local, one also has  $\Eins_E |\nabla \varphi(r \sqrt{\mathcal{L}}) \Eins_F |^2=0$ if $\dist(E,F)>r$.

Indeed, for every nonnegative, bounded and Lipschitz function $\chi_E$ supported in $E$ and $g\in {\mathcal D}$ supported in $F$, it follows from \eqref{def}, \eqref{eq:leibniz}, and \eqref{eq:carre} that
\begin{align*}
 \int_M \chi_E |\nabla \varphi(r \sqrt{\mathcal{L}}) g |^2 \, d\mu &  \leq  \int_M |\varphi(r\sqrt{\mathcal{L}})g|^2 |\nabla \chi_E|^2 \, d\mu \\
& \ + \int_M \chi_E |\varphi(r \sqrt{\mathcal{L}}) g| |\mathcal{L} \varphi(r \sqrt{\mathcal{L}}) g| \, d\mu.
\end{align*}
If $d(E,F)>\sqrt{r}$ then $\varphi(r\sqrt{\mathcal{L}})g=0$ in the support of $\chi_E$, so
$$  \int \chi_E |\nabla \varphi(r \sqrt{\mathcal{L}}) g |^2 \, d\mu=0.$$
This holds for every nonnegative, bounded Lipschitz function $\chi_E$ supported in $E$, hence 
$$ \int_E |\nabla \varphi(r \sqrt{\mathcal{L}}) g |^2 \, d\mu=0.$$

\begin{proof}[Proof of Proposition $\ref{prop:Caccioppoli}$]
Consider an even function $\varphi \in \mathcal{S}(\R)$ with $\supp \hat{\varphi} \subseteq [-1,1]$ and $\varphi(0)=1$. Consequently, $\varphi'(0)=0$, and $x\mapsto x^{-1}\varphi'(x) \in \mathcal{S}(\R)$ is even with Fourier support in $ [-1,1]$, cf. \cite[Lemma 6.1]{AMcM}.  Fix a ball $B_r$ of radius $r>0$, an exponent $q\in(1,p]$ and split
$$ f = \varphi(r\sqrt{\mathcal{L}}) f + (I- \varphi(r\sqrt{\mathcal{L}})) f.$$
Since $\varphi(0)=1$, one has
\begin{align*}
 I- \varphi(r\sqrt{\mathcal{L}}) = \int_0^r \sqrt{\mathcal{L}}\varphi'(s\sqrt{\mathcal{L}})  \, ds.
\end{align*}
Using the finite propagation speed  property applied to the functions $\varphi$ and $x\mapsto x^{-1}\varphi'(x)$, we have that both $\varphi(r\sqrt{\mathcal{L}})$ and $(r^2 \mathcal{L})^{-1}(1- \varphi(r\sqrt{\mathcal{L}}))$ satisfy the  propagation property at  speed $1$ and so propagate at a distance at most $r$. As we have seen above, the same stills holds by composing with the gradient. Hence, for $f\in \mathcal{D}$,
\begin{align}\label{eq:cc}
	\||\nabla f |\|_{L^q(B_r)}
		& \lesssim \||\nabla\varphi(r \sqrt{\mathcal{L}})|\|_{q\to q} \|f\|_{L^q(2B_r)}   \\
		& +  r^2 \||\nabla (1- \varphi(r\sqrt{\mathcal{L}}))(r^2 \mathcal{L})^{-1}| \|_{q\to q} \| \mathcal{L} f \|_{L^q(2B_r)}. \nonumber
\end{align}
Let us now estimate $\||\nabla\varphi(r \sqrt{\mathcal{L}})|\|_{q\to q}$ and $\||\nabla (1- \varphi(r\sqrt{\mathcal{L}}))(r^2 \mathcal{L})^{-1}| \|_{q\to q}$.
For $q>2$,  $(G_q)$ holds by interpolation between $(G_2)$ and $(G_p)$, and for  $q<2$, $(G_q)$ always holds as we already said. By writing the resolvent via the Laplace transform as $(1+r^2\mathcal{L})^{-1} = \int_0^{+\infty}  e^{-t(1+r^2\mathcal{L})} dt$, we deduce gradient bounds for the resolvent in $L^q$, that is
\begin{align*}
\| |\nabla (1+r^2\mathcal{L})^{-1} |\|_{q \to q} & \lesssim \int_0^{+\infty }e^{-t} \||\nabla e^{-tr^2\mathcal{L}}|\|_{q \to q} \,dt 
  \lesssim \int_0^{+\infty} \frac{e^{-t}}{r\sqrt{t}} \,dt  \lesssim r^{-1}. 
\end{align*}
Denote $\psi:=\varphi$ or $\psi:=x\mapsto (1-\varphi(x))/x^2$, and consider $\lambda(x)=\psi(x) (1+x^2)$. Hence
$$ \psi(r \sqrt{\mathcal{L}}) = (1+r^2\mathcal{L})^{-1} \lambda(r\sqrt{\mathcal{L}}) $$
and therefore
\begin{align*} 
\| |\nabla \psi(r \sqrt{\mathcal{L}})|\|_{q\to q} & \leq \| |\nabla (1+r^2\mathcal{L})^{-1}| \|_{q \to q} \| \lambda(r\sqrt{\mathcal{L}}) \|_{q \to q} 
 \lesssim r^{-1} \| \lambda(r\sqrt{\mathcal{L}}) \|_{q \to q}.
\end{align*}
Observe that $\lambda \in \mathcal{S}(\R)$ since $\varphi(0)=1$ and $\varphi'(0)=0$. 
By a functional calculus result (see e.g. \cite{DOS}, Theorem 3.1)
which relies on \eqref{UE}, we then have 
$$ \sup_{r>0} \ \| \lambda(r\sqrt{\mathcal{L}}) \|_{q \to q} \lesssim 1,$$
and consequently,
$$ \| |\nabla \psi(r \sqrt{\mathcal{L}})| \|_{q\to q} \lesssim r^{-1}.$$
Coming back to \eqref{eq:cc}, we obtain
\begin{align*}
	\||\nabla f |\|_{L^q(B_r)}	& \lesssim r^{-1} \| f\|_{L^q(2B_r)} +  r \| \mathcal{L} f \|_{L^q(2B_r)},
	\end{align*}
which is the claim.
\end{proof}

We are now ready to prove Proposition \ref{prop1}. To pass from $(DG_{2})$ to $(DG_{q})$ for $2<q\le p$ we shall use H\"older estimates as in \cite[Lemma 2.3]{HKim}.

\begin{proof}[Proof of Proposition $\ref{prop1}$] 
Let $B_R$ denote a ball of radius $R>0$, and let $f \in \mathcal{F}$. 
 Using  Lemma \ref{lemma:meyers} for $p=2$ and $(P_2)$,   we can write
\begin{equation}\label{holP}
\|f\|_{C^{\eta}(B_R)}  \lesssim \|f\|_{C^{\eta,2}(B_R)}  
  \lesssim \sup_{\tilde B \subset 6B_R} r(\tilde{B})^{1-\eta} \left(\aver{\tilde B} \left| \nabla f \right|^2 d\mu \right)^{{1/2}}.
\end{equation}
for any $\eta\in(0,1)$.  Now according to Proposition \ref{propDG}, $(P_2)$ yields $(DG_{2,\epsilon})$ for some $\epsilon\in(0,1)$.
Choose $\eta=1-\eps$. Then it follows from  \eqref{holP} together  with $(DG_{2,\epsilon})$  that
\begin{eqnarray*}
 \|f\|_{C^{1-\epsilon,2}(B_R)} 
  &\lesssim& \sup_{\tilde B \subset 6B_R} r^\eps(\tilde{B}) \left(\aver{\tilde B} \left| \nabla f \right|^2 d\mu \right)^{{1/2}}\\
 &\lesssim& R^{\eps} \left[ \left(\aver{24 B_R} \left| \nabla f \right|^2 d\mu \right)^{{1/2}}  + R \esssup_{x\in 24 B_R} |\mathcal{L} f(x)| \right]. 
\end{eqnarray*}
We then apply Jensen's inequality and obtain for $q\geq 2$
\begin{equation} \label{eq:hol}  \|f\|_{C^{1-\epsilon}(B_R)} \lesssim R^{\eps} \left[ \left(\aver{24 B_R} \left| \nabla f \right|^q d\mu \right)^{\frac{1}{q}}  + R \esssup_{x\in 24 B_R} |\mathcal{L} f(x)|\right]. \end{equation}

We now deduce $(DG_q)$ for $2<q\le p$ from \eqref{eq:hol} and $(G_p)$.

By $(G_p)$ and Proposition \ref{prop:Caccioppoli} one has \eqref{cacp} for $2<q\le p$. Replacing $f$ with  $f-\aver{B_r} f d\mu$
in \eqref{cacp} yields
\begin{equation}\label{cacm}
 \left( \aver{B_r} |\nabla f|^q d\mu \right)^{1/q}  \lesssim r^{-1 } \left(\aver{2 B_r} |f-\aver{B_r} f d\mu|^q d\mu \right)^{1/q}+ r \left(\aver{2 B_r} | \mathcal{L} f|^q d\mu \right)^{1/q}.
 \end{equation}
for every ball  $B_r$  with radius $r>0$ and $B_r \subset B_R$.

Now   we can write
\begin{align*}
 \left(\aver{2 B_r} |f-\aver{B_r} f d\mu|^q d\mu \right)^{1/q} & 
  \le\esssup_{x,y\in 2B_r}|f(x)-f(y)| \\& 
  \le r^{1-\epsilon} \|f\|_{C^{1-\epsilon}(B_{2r})} \\ &\leq r^{1-\epsilon} \|f\|_{C^{1-\epsilon}(B_{2R})}
\end{align*}
and \eqref{eq:hol} yields
\begin{align*}
 \left(\aver{2 B_r} |f-\aver{B_r} f d\mu|^q d\mu \right)^{1/q} 
& \lesssim r^{1-\epsilon} R^{\epsilon} \left[\left(\aver{48 B_R} |\nabla f|^q d\mu \right)^{1/q} +  R \esssup_{x\in 48 B_R} |\mathcal{L} f(x)| \right].
\end{align*}
Consequently
\begin{eqnarray*}
\left( \aver{B_r} |\nabla f|^q d\mu \right)^{1/q}  & \lesssim&  \left(\frac{R}{r}\right)^{\epsilon} \left[\left(\aver{48 B_R} |\nabla f|^q d\mu \right)^{1/q} 
  + R \esssup_{x\in 48 B_R} |\mathcal{L} f(x)|\right]\\
  &&+ r \left(\aver{2 B_r} | \mathcal{L} f|^q d\mu \right)^{1/q}\\
   & \lesssim&  \left(\frac{R}{r}\right)^{\epsilon} \left[\left(\aver{48 B_R} |\nabla f|^q d\mu \right)^{1/q} 
  + R \esssup_{x\in 48 B_R} |\mathcal{L} f(x)|\right],
  \end{eqnarray*}
which gives easily $(DG_{q,\epsilon})$.
\end{proof}

Let us finish this section by noting a  consequence of \cite[Theorem 0.4]{AC} and Proposition \ref{prop1} on the De Giorgi property.
\begin{coro}
Let $(M,d,\mu, {\mathcal E})$  be  a doubling metric  measure   Dirichlet space with a ``carr\'e du champ'' satisfying    $(P_2)$. 
 Then there exists $\eps>0$ such that  $({DG}_p)$  holds for every $p\in[2,2+\eps)$.
\end{coro}

\section{$L^2$ Poincar\'e inequality through harmonic functions}
\label{sec:poi}



\begin{theorem} \label{thm:poincare} Let $(M,d,\mu, {\mathcal E})$  be  a doubling metric  measure   Dirichlet space with a ``carr\'e du champ'' satisfying \eqref{FKR}. Suppose that 
$$ \left( \aver{B} | u- \aver{B} u \, d\mu|^2 \, d\mu \right)^{1/2} \lesssim r \left( \aver{B} |\nabla u|^2\, d\mu \right)^{1/2}$$
for every ball $B$ of radius $r>0$ and every function $u\in {\mathcal D}$ harmonic on $4B$.
Then 
$$ \left( \aver{B} | f- \aver{B} f \, d\mu |^2 \, d\mu \right)^{1/2} \lesssim r \left( \aver{4B} |\nabla f|^2\, d\mu \right)^{1/2}$$
for every ball $B$ of radius $r>0$ and every function $f\in {\mathcal D}$.
\end{theorem}

\begin{proof} Let $f\in {\mathcal D}$ and $B$ be a ball of radius $r>0$. Recall that we denote the $L^2$-oscillation of $f$ over $B$ by
$$ \osc_B(f)= \left( \aver{B} | f- \aver{B} f \, d\mu|^2 \, d\mu \right)^{1/2}.$$
Since \eqref{FKR} holds, one can invoke  Lemma \ref{lem:LM}: there exists $u\in {\mathcal F}$ harmonic on $4B$ such that $f-u\in {\mathcal F}$ is supported in the ball $4B$. 
One may write, using the triangle inequality,
\begin{align*} 
\osc_{B}(f) & \leq \osc_{B}(u) +  \osc_{B}(f-u). 
\end{align*}
Using the assumption for $u$, we have
$$  \osc_{B}(u) \lesssim r \left( \aver{B} |\nabla u|^2\, d\mu \right)^{1/2} \lesssim r \left( \aver{B} |\nabla (f-u)|^2\, d\mu \right)^{1/2} + r \left( \aver{B} |\nabla f|^2\, d\mu \right)^{1/2}.$$
Moreover, since $f-u$ is supported on $4B$, then by \eqref{FKR}, 
$$\osc_{B}(f-u)\lesssim \left( \aver{4B} |f-u|^2 \, d\mu \right)^{1/2} \lesssim r \left( \aver{4B} |\nabla (f-u)|^2 \, d\mu \right)^{1/2}.$$
As a consequence, it follows that
\begin{equation} \osc_B(f) \lesssim r \left( \aver{4B} |\nabla (f-u)|^2\, d\mu \right)^{1/2} + r \left( \aver{B} |\nabla f|^2\, d\mu \right)^{1/2}. \label{eq:1} \end{equation}
Finally, by combining \eqref{eq:grad} and \eqref{eq:1}, we deduce that
$$ \osc_B(f) \lesssim  r \left( \aver{4B} |\nabla f|^2\, d\mu \right)^{1/2},$$
which is the claim.
\end{proof}

Let us recall the following equivalence between weak and strong Poincar\'e inequalities (combining \cite[Theorem 3.1]{HaKo} and \cite{KZ}).
\begin{theorem} \label{thm:wpoincare} Let $(M,d,\mu, {\mathcal E})$  be  a doubling metric  measure   Dirichlet space with a ``carr\'e du champ''. Let $p\in(1,\infty)$, and let $f\in {\mathcal D}$. The (strong) $L^p$ Poincar\'e inequality $(P_p)$ for $f$ is equivalent to the weak version: there exists $\lambda>1$ such that  
\begin{equation}\tag{$w$-$P_p$}
 \left( \aver{B} \left| f - \aver{B} f d\mu \right|^p \, d\mu \right)^{1/p} \lesssim r \left(\aver{\lambda B} |\nabla f|^p d\mu \right)^{1/p}, 
 \label{w-Pp}
\end{equation}
where $B$ ranges over balls in $M$ of radius $r$.
\end{theorem}

We then state the following main result of this section:

\begin{theorem} \label{thm:gpppp2} Let $(M,d,\mu, {\mathcal E})$  be  a doubling metric  measure   Dirichlet space with a ``carr\'e du champ''. For every $p>2$, the combination $(G_p)$ with $(P_p)$ yields $(P_2)$, therefore the heat semigroup satisfies the two-sided Gaussian estimates \eqref{UE} and \eqref{LE}.
\end{theorem}

Note that the  last statement  relies on the theorem of Saloff-Coste of which we have given a simplified version in Section \ref{section:2}.

%

Before proving the theorem, we need the following proposition, which relies on the  self-improving property of  reverse H\"older inequalities from Theorem \ref{thm:imp}. 

\begin{proposition}\label{prop} 
Let $(M,d,\mu, {\mathcal E})$ be a doubling metric  measure   Dirichlet space with a ``carr\'e du champ''.
Assume that for some $p_0\in(2,\infty)$, $(P_{p_0})$ and $(G_{p_0})$  hold. Then for every function $f\in {\mathcal D}$ and every ball $B$ of radius $r>0$, we have
\begin{align*}
\left(\aver{B_r} | f - \aver{B_r} f \, d\mu |^2 \, d\mu \right)^{1/2} \lesssim r \left(\aver{2B_r} | \nabla f |^2 \, d\mu \right)^{1/2}  + r^2 \| \mathcal{L}f \|_{L^\infty(4B_r)}.
\end{align*}
\end{proposition}

Let us first recall the following folklore result (see for instance \cite[Theorem 5.1, 1.]{HaKo}, \cite[Theorem 2.7]{FPW} for similar statements).

\begin{lemma}\label{met} Let $(M,d,\mu, {\mathcal E})$  be  a  doubling metric  measure   Dirichlet space with a ``carr\'e du champ''.  Assume  that
$(P_p)$ holds for some $1\le p<+\infty$. Then if $q\in (p,+\infty)$ is such that  $\nu\left(\frac{1}{p}-\frac{1}{q}\right)\le 1$,
the following Sobolev-Poincar\'e inequality holds:
\begin{equation}\label{Ppq}\tag{$P_{p,q}$}
 \left(\aver{B_r} \left| f-\aver{B_r}f\,d\mu\right|^{q} d\mu \right)^{1/q}\lesssim r \left(\aver{B_r} |\nabla f|^{p} d\mu \right)^{1/p},
  \end{equation}
for all $f\in\mathcal{F}$, $r>0$, and all balls $B_r$ with radius $r$.
\end{lemma}

\begin{proof} This result is well-known if $1<p<\nu$ (see \cite[Theorem 5.1, 1.]{HaKo}). Note that the truncation property holds in our setting (see \cite[Section 3,(o)]{mosco}). If $p>\nu$, a stronger  $L^\infty$ inequality is true (\cite[Theorem 5.1, 2.]{HaKo}), and one gets the above by integration. If $p=\nu$, one deduces the claim from \cite[Theorem 5.1, 3.]{HaKo} by bounding from above the function $t\to t^q$  by an exponential. 
\end{proof}

\begin{proof}[Proof of Proposition $\ref{prop}$]
Proposition \ref{prop:Caccioppoli}   yields
\begin{equation*}
 \left(\aver{B_r} | \nabla f|^{p_0} \, d\mu \right)^{1/p_0}  \lesssim \frac{p_0\text{-}\osc_{2B_r}(f)}{r}+ r\left(\aver{2B_r} | \mathcal{L}f|^{p_0} \, d\mu \right)^{1/p_0}
\end{equation*}
for every ball $B_r$ of radius $r$ and every $f\in\mathcal{D}$.
From \cite{KZ} we know that $(P_{p_0})$ self-improves into $(P_{p_0-\epsilon})$ for some $\epsilon>0$. Then , according to Lemma \ref{met},  if $\epsilon$ is small enough, $(P_{p_0-\epsilon})$ self-improves again into the Sobolev-Poincar\'e inequality $(P_{p_0,p_0-\epsilon})$
which yields
$$  p_0\text{-}\osc_{2B_r}(f) \lesssim r \left(\aver{2B_r} |\nabla f|^{p_0-\epsilon} d\mu \right)^{1/(p_0-\epsilon)} .$$
We therefore have 
\begin{equation}\label{psl}
 \left(\aver{B_r} | \nabla f|^{p_0} \, d\mu \right)^{1/p_0}  \lesssim \left(\aver{2B_r} |\nabla f|^{p_0-\epsilon}\, d\mu \right)^{1/(p_0-\epsilon)}+ r\left(\aver{2B_r} | \mathcal{L}f|^{p_0} \, d\mu \right)^{1/p_0}.
\end{equation}
Let us apply Theorem \ref{thm:imp} to the functional 
$$ a(B) := r \left(\aver{2B} | \mathcal{L}f|^{p_0} \, d\mu \right)^{1/p_0},$$
which is regular, uniformly with respect to $f$. This gives
\begin{align*}
 \left(\aver{B_r} | \nabla f|^{p_0} d\mu \right)^{1/p_0} \lesssim \left(\aver{2B_r} |\nabla f|^2\, d\mu \right)^{1/2}+ r\left(\aver{4B_r} | \mathcal{L}f|^{p_0} \, d\mu \right)^{1/p_0}.
\end{align*}
Using H\"older's inequality and $(P_{p_0})$, we deduce that
\begin{align*}
\left(\aver{B_r} | f - \aver{B_r} f \,d\mu |^2 \, d\mu \right)^{1/2} & \leq \left(\aver{B_r} | f - \aver{B_r} f \,d\mu |^{p_0} \, d\mu \right)^{1/p_0} \\
& \lesssim  r\left(\aver{B_r} | \nabla f|^{p_0} d\mu \right)^{1/p_0} \\
& \lesssim  r\left(\aver{2B_r} |\nabla f|^{2}\, d\mu \right)^{1/2}+ r^2\left(\aver{4B_r} | \mathcal{L}f|^{p_0} \, d\mu \right)^{1/p_0}.
\end{align*}
Bounding the last term by $r^2 \| \mathcal{L}f \|_{L^\infty(4B_r)}$ then gives the desired estimate.
\end{proof}

We may now prove Theorem $\ref{thm:gpppp2}$:

\begin{proof}[Proof of Theorem \ref{thm:gpppp2}] First we already have seen that the combination $(G_p)$ with $(P_p)$ implies \eqref{due} (Proposition \ref{prop:ppgp}) and it is known that it yields \eqref{FKR}. \\
By Proposition \ref{prop}, the combination $(G_p)$ with $(P_p)$ (for some $p>2$) implies a weak Poincar\'e inequality $(w-P_2)$ for all harmonic functions. Then from Theorem \ref{thm:poincare}, we deduce a weak Poincar\'e inequality for every function of ${\mathcal D}$.
We conclude to the strong Poincar\'e inequality by Theorem \ref{thm:wpoincare}.
\end{proof}

An obvious by-product of Theorem  \ref{thm:gpppp2} is the following monotonicity property for $(G_p)+(P_p)$.

\begin{coro}\label{mono}  Let $(M,d,\mu, {\mathcal E})$ be a metric  measure   Dirichlet space with a ``carr\'e du champ'' satisfying \eqref{d}. Then $(G_p)+(P_p)$ for some $p>2$ implies  $(G_q)+(P_q)$ for all $q\in [2,p)$. 
\end{coro}

\appendix

\section{From Poincar\'e to De Giorgi} \label{AppDG}

In this appendix, we give a self-contained proof for the fact that, under \eqref{d}, \eqref{P2} implies the De Giorgi property ($DG_2$), as stated in Proposition \ref{propDG}.  
 One method is to use the  elliptic Moser iteration process  from \cite{Moser}, see for instance \cite[Sections 5 and 6]{AC2}. The proof is given there in a discrete time and space setting,
 but adapts to our current setting. Another proof in \cite{AC2} by-passes the difficult part of the Moser iteration process, namely the John-Nirenberg lemma,
 and uses instead  the self-improvement property of Poincar\'e inequalities. This is the one we will present here.

In this appendix, we will assume for simplicity that $\diam(M)=+\infty$. If on the contrary  $\diam (M)<+\infty$, it is enough to assume that $R \leq \delta\, \diam(M)$,
where $\delta$ is the parameter that has to appear in \eqref{FK} in that case, and to use doubling at the end of the argument.

We first need to state  a  version of the  Caccioppoli inequality \eqref{cacp} in $L^2$, but for subharmonic functions. 

\begin{lemma} \label{prop:Cacc-harmonic}
Let $(M,d,\mu, {\mathcal E})$  be  a doubling metric  measure   Dirichlet space with a ``carr\'e du champ''.
For every $x \in M$, $0<r<R$ and every 
$u\in \mathcal {D}$ with $u\mathcal{L}u \leq 0$ on $B(x,R)$, one has
\begin{equation}\label{cac2}
\int_{B(x,r)} |\nabla u|^2 \,d\mu \lesssim \frac{1}{(R-r)^2} \int_{B(x,R)} \left|u\right|^2 \,d\mu.
\end{equation}
\end{lemma}

\begin{proof}  Consider a  function $\chi$ belonging to $\mathcal{F}$, supported on $B(x,R)$, with values in $[0,1]$, and such that $\chi\equiv1$ on $B(x,r)$ and $\| |\nabla \chi| \|_\infty \lesssim (R-r)^{-1}$. Such a function can easily be built in our setting from the metric $d$ (see for instance \cite[Section 2.2.6]{GSC} for details).



Since  $u\mathcal{L}u \leq 0$ on $B(x,R)$ one may write
\begin{align*} 
0 & \leq - \int_M \chi^2 u \mathcal{L}u \, d\mu = -\mathcal{E}(\chi^2 u, u ) \\
& =-\frac{1}{2} \mathcal{E}(\chi^2  , u^2)- \int_M \chi^ 2 |\nabla u|^2 d\mu,
\end{align*}
where one uses \eqref{def}.
Consequently, by   \eqref{eq:leibniz} and \eqref{eq:carre},
$$ I:=\int_M \chi^ 2 |\nabla u|^2 d\mu \leq \frac{1}{2}\left|\mathcal{E}(\chi^2  , u^2)\right|\leq  2\int_M \chi |u| |\nabla \chi| |\nabla u| \, d\mu$$
and we deduce by Cauchy-Schwarz that
$$ I \lesssim \frac{1}{R-r} \int_M |u| |\nabla u| \chi \, d\mu \lesssim \frac{1}{R-r} \sqrt{I} \left[ \int_{B(x,R)} |u|^2  \, d\mu \right]^ {1/2},$$
which yields \eqref{cac2}.  
\end{proof}

 First, the relative Faber-Krahn  inequality \eqref{FK} (see Section \ref{section:2} for a definition) implies an $L^2$ mean value property for harmonic functions. 

\begin{proposition} \label{propMax}
Let $(M,d,\mu, {\mathcal E})$  be  a doubling metric  measure   Dirichlet space with a ``carr\'e du champ'' satisfying \eqref{UE}. Assume \eqref{FK}. Then one has, for all $R>0$, $x_0 \in M$, and $u \in \mathcal{F}$ harmonic in $B(x_0,R)$,
\[
	\max_{x \in B(x_0,R/2)} u(x) 
		\lesssim \left(\aver{B(x_0,R)} u_{+}^2 \,d\mu\right)^{1/2}.
\]
\end{proposition}

\begin{proof}
Let $x_0\in M$, $R>0$ and $u\in  \mathcal{F}$ such that  $\mathcal{L}u=0$ in $B(x_0,R)$.
For $h,r>0$ define
\begin{align*}
	M(r)&=\max_{x \in B(x_0,r)} u(x), \qquad 
	m(r)=\min_{x \in B(x_0,r)} u(x),\\
	A(h,r) &= \{x \in B(x_0,r); \ u(x) \geq h\}, \qquad
	a(h,r)= \mu(A(h,r)).
\end{align*}

Consider  $\rho<r\leq 2 \rho$ with $r\leq R$. Let $\chi$ be a Lipschitz function supported on $B(x_0,r)$, that equals $1$ on $B(x_0,\rho)$ and such that $\|\nabla \chi\|_\infty \lesssim (r-\rho)^{-1}$.
Applying \eqref{FK}  to $\chi (u-h)_+$ (which is supported on $B(x_0,r)$) for $h \in \R$ gives 
\begin{align*}
\lefteqn{\int_{B(x_0,\rho)} (u-h)^2_+ \,d\mu  \leq \int \chi^ 2 (u-h)^2_+ \,d\mu } & & \\
 & & \lesssim r^2 \left(\frac{a(h,r)}{V(x_0,r)}\right)^{2\beta} \left[\int_{B(x_0,r)} |\nabla(u-h)_+|^2 \,d\mu + \frac{1}{(r-\rho)^2} \int_{B(x_0,r)} (u-h)_+ ^2 \,d\mu\right],
\end{align*}
where we use the Leibniz rule for the gradient.
 Then, since $u-h$ is harmonic, that is $\mathcal{L}(u-h)=0$ in $B(x_0,R)$, by \cite[Lemma 2]{ST1} we deduce that $(u-h)_+$ is nonnegative and subharmonic, that is  $\mathcal{L}(u-h)_+ \leq 0$, in $B(x_0,R)$. So $(u-h)_+\mathcal{L}(u-h)_+ \leq 0$ in $B(x_0,R)$.
We then deduce from \eqref{cac2} that 
\begin{align*}
	\int_{B(x_0,\rho)} (u-h)^2_+ \,d\mu  
	& \lesssim \frac{r^2}{(r-\rho)^2} V(x_0,r)^{-2\beta} a(h,r)^{2\beta} \int_{B(x_0,r)} (u-h)_+^2 \,d\mu \\
	& \lesssim \frac{\rho^2}{(r-\rho)^2} V(x_0,\rho)^{-2\beta} a(h,r)^{2\beta} \int_{B(x_0,r)} (u-h)_+^2 \,d\mu,
\end{align*}
where we used the doubling property and $\rho\simeq r$.
Set
\begin{align*}
	u(h,\rho)= \int_{B(x_0,\rho)} (u-h)^2_+ \,d\mu  = \int_{A(h,\rho)} (u-h)^2 \,d\mu.
\end{align*}
One has 
\begin{align} \label{eq:maxpr1}
	u(h,\rho) \lesssim \frac{\rho^2}{(r-\rho)^2} V(x_0,\rho)^{-2\beta} a(h,r)^{2\beta}  
	u(h,r).
\end{align}
Moreover, for $h>k$,
\begin{align*}
	(h-k)^2 a(h,r) \leq \int_{A(h,r)} (u-k)^2 \,d\mu \leq \int_{A(k,r)} (u-k)^2 \,d\mu,
\end{align*}
that is
\begin{align} \label{eq:maxpr2}
	a(h,r) \leq \frac{1}{(h-k)^2} u(k,r). 
\end{align}
Then \eqref{eq:maxpr1} and \eqref{eq:maxpr2} yield for $\rho<r\leq 2\rho$
\begin{align} \label{eq:maxpr3}
	u(h,\rho) \leq C \frac{\rho^2}{(r-\rho)^2} V(x_0,\rho)^{-2\beta} \frac{1}{(h-k)^{4\beta}} u(k,r)^{1+2\beta}.
\end{align}
Set $k_n= (1-\frac{1}{2^n})d$ and $\rho_n=\frac{R}{2}(1+\frac{1}{2^n})$, $n \in \N$, where 
\begin{align*}
	d=C^{\frac{1}{4\beta}} 2^{\frac{|\nu\beta -1|}{2\beta}} 2^{\frac{1}{4\beta^2}+\frac{2}{\beta}+1} V(x_0,R)^{-{1/2}} u(0,R)^{\frac12}.
\end{align*}
Inequality \eqref{eq:maxpr3} (which can be applied since $\rho_{n+1}<\rho_n\leq 2 \rho_{n+1}$) yields
\begin{align*}
	u(k_{n+1},\rho_{n+1})
		\leq C \frac{2^{2(n+2)+4\beta (n+1)} \rho_{n+1}^2}{R^2 d^{4\beta}} V(x_0,\rho_{n+1})^{-2\beta} u(k_n,\rho_n)^{1+2\beta},
\end{align*}
therefore by \eqref{dnu}, 
\begin{align*} 
	u(k_{n+1},\rho_{n+1})
		\leq C 2^{2|\nu\beta -1|} \frac{2^{2(n+2)+4 \beta(n+1)}}{d^{4\beta}} V(x_0,R)^{-2\beta} u(k_n,\rho_n)^{1+2\beta}.
\end{align*}
Due to the definition of $d$, this yields
\begin{align} \label{eq:maxpr4}
	u(k_{n+1},\rho_{n+1})
		\leq 2^{2(n-2)+4 \beta n -\frac{1}{\beta}} u(0,R)^{-2\beta} u(k_n,\rho_n)^{1+2\beta}.
\end{align}
From \eqref{eq:maxpr4} one proves by induction that
\begin{align*}
	u(k_n,\rho_n) \leq \frac{u(0,R)}{2^{(2+\beta^{-1}) n}}, \qquad \forall n \in \N.
\end{align*}
By letting $n$ go to infinity, one concludes that $u(d,R/2) =0$.
This means that for all $x \in B(x_0,R/2)$,
\begin{align*}
	u(x) \leq d=C\left(\frac{1}{V(x_0,R)} \int_{A(0,R)} u^2 d\mu \right)^{1/2}.
\end{align*}
\end{proof}

\bigskip

To go further, we will need to use scaled Poincar\'e inequalities and their consequences. Assuming \eqref{P2}, we know that there exists $\eps>0$ such that $(P_{2-\eps})$ holds \cite{KZ}. We will be working with the modified version 
\begin{align}
	\tag{$\widetilde{P}_{2-\eps}$}
 \left( \aver{B} \left| f \right|^{2-\eps} d\mu \right)^{\frac{1}{2-\eps}} \lesssim r \frac{|B|}{|\{y \in B; \ f(y)=0\}|}\left(\aver{B} |\nabla f|^{2-\eps} d\mu \right)^{\frac{1}{2-\eps}}, \label{modPp}
\end{align}
where $f$ ranges in ${\mathcal D}$ and $B$  in balls in $M$ of radius $r$. \\
The inequality $(\widetilde{P}_{2-\eps})$ is a consequence of $(P_{2-\eps})$, as can be seen by checking the inequality 
\begin{align*}
	\left(\int_B |f|^{2-\eps} \,d\mu \right)^{\frac{1}{2-\eps}} \leq C_\eps \frac{\abs{B}}{|\{y \in B; \ f(y)=0\}|} \left(\int_B |f-\aver{B} f \,d\mu|^{2-\eps} \,d\mu \right)^{\frac{1}{2-\eps}}. 
\end{align*}
To prove the latter, abbreviate $B_{N(f)}:=\{y \in B; \ f(y)=0\}$ and write
\begin{align*}
&\left(\int_B |f|^{2-\eps} \,d\mu \right)^{\frac{1}{2-\eps}}  = \left(\int_{B\setminus B_{N(f)}} |f|^{2-\eps} \,d\mu \right)^{\frac{1}{2-\eps}}\\
\leq & \left(\int_{B\setminus B_{N(f)}}  |f-\aver{B} f \,d\mu|^{2-\eps} \,d\mu \right)^{\frac{1}{2-\eps}} +   \left(\int_{B\setminus B_{N(f)}}  \left|\aver{B} f \,d\mu\right|^{2-\eps} \,d\mu \right)^{\frac{1}{2-\eps}}  \\
=  &  \left(\int_{B\setminus B_{N(f)}}  |f-\aver{B} f\, d\mu|^{2-\eps} \,d\mu \right)^{\frac{1}{2-\eps}}+ |B\setminus B_{N(f)}|^{\frac{1}{2-\eps}} \;  \left|\aver{B} f \,d\mu\right|   \\      
\le &   \left(\int_{B}  |f-\aver{B} f \,d\mu|^{2-\eps} \,d\mu \right)^{\frac{1}{2-\eps}}+ \left(\frac{|B\setminus B_{N(f)}|}{|B|}\right)^{\frac{1}{2-\eps}}    \left(\int_B |f|^{2-\eps} \,d\mu \right)^{\frac{1}{2-\eps}}. 
\end{align*}
From this, we deduce
\begin{align*}
	 \left(\int_B |f|^{2-\eps} \,d\mu \right)^{\frac{1}{2-\eps}} 
	&\leq \left(1- \left(\frac{|B\setminus B_{N(f)}|}{|B|}\right)^{\frac{1}{2-\eps}}    \right)^{-1} 
	\left(\int_{B}  |f-\aver{B} f \,d\mu|^{2-\eps} \,d\mu \right)^{\frac{1}{2-\eps}}\\
	& \lesssim \frac{|B|}{|B_{N(f)}|}
	\left(\int_{B}  |f-\aver{B} f \,d\mu|^{2-\eps} \,d\mu \right)^{\frac{1}{2-\eps}},
\end{align*}
where in the last step we have used the elementary inequality $1-(1-x)^{1/p} \geq \frac{x}{p}$ for $x \in [0,1]$ and $p \in [1,+\infty)$.

\begin{lemma} \label{lemmaDG}
Let $(M,d,\mu, {\mathcal E})$  be  a doubling metric  measure   Dirichlet space with a ``carr\'e du champ'' satisfying  \eqref{UE} and \eqref{FK}. Let $\eps \in (0,1]$ and assume $(\widetilde{P}_{2-\eps})$. Let $x_0 \in M$, $R>0$ and $u \in \mathcal{F}$ harmonic in $B(x_0,R)$. Set 
\[
	k_i = M(R)-\left(\frac{M(R)-m(R)}{2^{i+1}}\right), \quad i \in \N. 
\]
Assume that $a(k_0,R/2) \leq \frac12 V(x_0,R/2)$. Then for all integer $i \in \N^*$
\[
	\frac{a(k_i,R/2)}{V(x_0,R/2)} \leq C i^{-\eps/2},
\]
where $C$ does not depend on $x_0$, $R$, or $u$. 
\end{lemma}

\begin{proof}
For $h>k>k_0$, set $v=(u-k)_+ \wedge (h-k)$. By assumption,
\begin{align*}
	|\{x \in B(x_0,R/2); \ v(x)=0\}|
	&= |B(x_0,R/2) \setminus A(k,R/2)| \\
	&\geq |B(x_0,R/2) \setminus A(k_0,R/2)|
	\geq {1/2} V(x_0,R/2).
\end{align*}
The Poincar\'e inequality $(\widetilde{P}_{2-\eps})$  therefore yields
\begin{align*}
	\int_{B(x_0,R/2)} \abs{v}^{2-\eps} \,d\mu 
	\lesssim  R^{2-\eps} \int_{B(x_0,R/2)} |\nabla v|^{2-\eps} \,d\mu .
\end{align*}
Hence
\begin{align*}
	(h-k)^{2-\eps} a(h,R/2) \lesssim R^{2-\eps} \int_{B(x_0,R/2)} |\nabla v|^{2-\eps} \,d\mu.
\end{align*}
Now
\begin{align*}
	(h-k)^{2-\eps} a(h,R/2) \lesssim R^{2-\eps} \int_{A(k,R/2) \setminus A(h,R/2)} |\nabla u|^{2-\eps} \,d\mu.
\end{align*}
By H\"older,
\begin{align*}
	(h-k)^{2-\eps} a(h,R/2)
	\lesssim R^{2-\eps} (a(k,R/2) - a(h,R/2))^{\eps/2} \left(\int_{A(k,R/2) \setminus A(h,R/2)} |\nabla u|^2\,d\mu \right)^{\frac{2-\eps}{2}}.
\end{align*}
Now
\begin{align*}
	\int_{A(k,R/2) \setminus A(h,R/2)} |\nabla u|^2 \,d\mu 
	\leq \int_{A(k,R/2)} |\nabla u|^2 \,d\mu 
	= \int_{B(x_0,R/2)} |\nabla (u-k)_+|^2 \,d\mu.
\end{align*}
As we already observed at the beginning of the proof of Proposition \ref{propMax} in a similar situation, $(u-k)_+$ is subharmonic,  therefore \eqref{cac2} yields
\begin{align*}
	\int_{B(x_0,R/2)} |\nabla (u-k)_+|^2 \,d\mu
	\lesssim R^{-2} \int_{B(x_0,R)} (u-k)_+^2 \,d\mu 
	\lesssim R^{-2} V(x_0,R) (M(R)-k)^2.
\end{align*}
Thus
\begin{align*}
	(h-k)^{2-\eps} a(h,R/2)
	\lesssim V(x_0,R)^{\frac{2-\eps}{2}}(M(R)-k)^{2-\eps} (a(k,R/2)-a(h,R/2))^{\frac{\eps}{2}}.
\end{align*}
Since $k_i-k_{i-1} = \frac{M(R)-k_0}{2^i}$ and $M(R) - k_{i-1} = \frac{M(R)-k_0}{2^{i-1}}$, the above inequality yields
\begin{align*}
	a(k_i,R/2)^{\frac{2}{\eps}}
	\lesssim V(x_0,R)^{\frac{2-\eps}{\eps}} (a(k_{i-1},R/2)-a(k_i,R/2)).
\end{align*}
Using the fact that $a(k_i,R/2)$ is non-increasing in $i$, one obtains
\begin{align*}
	ia(k_i,R/2)^{\frac{2}{\eps}}
	&\leq \sum_{j=1}^i a(k_j,R/2)^{\frac{2}{\eps}}
	\lesssim V(x_0,R)^{\frac{2-\eps}{\eps}} (a(k_0,R/2) - a(k_i,R/2))\\
	& \leq V(x_0,R)^{\frac{2-\eps}{\eps}} a(k_0,R/2).
\end{align*}
Hence
\begin{align*}
	\frac{a(k_i,R/2)}{V(x_0,R)} 
	\lesssim \left(i^{-1} \frac{a(k_0,R/2)}{V(x_0,R)}\right)^{\frac{\eps}{2}}
	\leq \left(\frac{i^{-1}}{2}\right)^{\frac{\eps}{2}}.
\end{align*}
\end{proof}

We are now in a position to deduce the elliptic regularity estimate \eqref{eq:ER}  introduced in Theorem \ref{thm:2} from $(P_2)$.

\begin{proposition} \label{propER}
 Let $(M,d,\mu, {\mathcal E})$  be  a doubling metric  measure   Dirichlet space with a ``carr\'e du champ'' satisfying    $(P_2)$.  Then 
 \eqref{eq:ER}  holds.
 \end{proposition}

\begin{proof} Recall first that  \eqref{FK} holds according to Remark \ref{PFK}.
Fix $x_0 \in M$, $R >0$, and let $u \in \mathcal{F}$ be harmonic in $B(x_0,R)$. Let $r \in (0,R/2]$. 
By applying Proposition \ref{propMax} to $u-K$ for any $K\leq M(2R)$, one obtains
\begin{equation} M(R/2)-K  \lesssim (M(2R)-K) \left( \frac{a(K,R)}{V(x_0,R)}\right)^{1/2}. \label{eq:hi} \end{equation}
According to \eqref{eq:hi} applied in $B(x_0,r)$ with $K=k_i:=k_i(2r)=M(2r)-\frac{M(2r)-m(2r)}{2^{i+1}}$, there exists a constant $C$, independent of the main parameters, such that 
\[
	M(r/2) \leq k_i(2r) +C(M(2r)-k_i(2r)) \left(\frac{a(k_i,r)}{V(x_0,r)}\right)^{1/2}.
\]
Assume that $a(k_0,r) \leq \frac{1}{2}V(x_0,r)$, otherwise work with $-u$. 
Since $(P_2)$  implies $(\widetilde{P}_{2-\eps})$  as we already pointed out, we can apply Lemma \ref{lemmaDG}. This yields 
\[
	\frac{a(k_i,r)}{V(x_0,r)} \leq C i^{-\eps/2},
\]
therefore one can choose $i$ large enough so that
\[
	C \left(\frac{a(k_i,r)}{V(x_0,r)}\right)^{1/2} \leq \frac{1}{2}. 
\]
One obtains 
\[
	M(r/2) \leq M(2r)- \frac{1}{2^{i+2}}(M(2r)-m(2r)),
\]
hence
\[
	M(r/2) -m(r/2) \leq (M(2r)-m(2r)) \left(1-\frac{1}{2^{i+2}}\right).
\]
Set $\omega(r)=M(r)-m(r)$. One has 
\[
	\omega(r/2) \leq \eta \omega(2r), \qquad \forall r \in (0,R/2],
\]
where $\eta=1-\frac{1}{2^{i+2}} \in (0,1)$. It follows that there exist $C,\alpha >0$ such that
\[
	\omega(\rho) \leq C \left(\frac{\rho}{R}\right)^\alpha \omega(R/2), \qquad \forall \rho \in (0,R/2].
\]
In particular,
\[
	\abs{u(x)-u(y)} \leq C' \left(\frac{d(x,y)}{R}\right)^\alpha \max_{B(x_0,R/2)} \abs{u}, \qquad \forall x,y \in B(x_0,R/2).
\]
Now, it follows easily from Proposition \ref{propMax} that
\[
	\max_{B(x_0,R/2)} \abs{u} \lesssim 2\text{-}\osc_{B(x_0,R)}(u),
\]
hence the claim.
\end{proof}

We can now prove Proposition \ref{propDG}. This will be done in two steps: first \eqref{eq:ER} yields  a De Giorgi property for harmonic functions similar to  \eqref{eq:dg}. Second one derives the full $(DG_2)$ by classical $L^2$ techniques.

\begin{proof}[Proof of Proposition $\ref{propDG}$]
Consider a function $u \in \mathcal{F}$ harmonic in $B_R=B(x_0,R)$. By Proposition \ref{propER}, we have \eqref{eq:ER}, and we have seen at the beginning of the proof of Theorem $\ref{thm:2}$ that this implies 
$$ \osc_{B_r}(u) \lesssim \left(\frac{r}{R}\right)^{\alpha} \osc_{B_R}(u),$$
for $B_r=B(x_0,r)$ and $0<r\le R$.
Using  the Caccioppoli inequality \eqref{cac2} and \eqref{P2}, we  obtain
\begin{align} \label{eq:DG-harm}
	\left(\aver{B(x_0,r)} |\nabla u|^2\,d\mu\right)^{1/2}
	\lesssim \left(\frac{R}{r}\right)^{1-\alpha} \left(\aver{B(x_0,R)} |\nabla u|^2\,d\mu\right)^{1/2}
\end{align}
for $0<r\le R/2$.
If $R/2\leq r\leq R$ then the inequality still holds by \eqref{d}. 

We can now deduce ($DG_2$) as follows (cf. \cite[Theorem 1.1]{Gia0} or \cite[Theorem 3.6]{Aus}). Let $f\in {\mathcal D}$,  $x_0 \in M$, and $R>0$. As in the proof of Theorem \ref{thm:2},  from Lemma \ref{lem:LM} (since \eqref{FK} follows from ($P_2$) let consider $u\in {\mathcal F}$ be harmonic on $B(x_0,R)$ such that $f-u\in {\mathcal F}$ is supported in the ball $B(x_0,R)$.  From \eqref{def}, \eqref{eq:leibniz}, and \eqref{eq:carre}, we deduce
\begin{equation}\label{fillon}
\||\nabla u|\|_{L^2(B(x_0,R))} \lesssim \||\nabla f|\|_{L^2(B(x_0,R))} + R \|\mathcal{L} f\|_{L^2(B(x_0,R))}.
\end{equation}
Using triangle inequality, then \eqref{d} and \eqref{eq:DG-harm}, write, for $0<r\le R$,
\begin{eqnarray*}
 \left(\aver{B(x_0,r)} |\nabla f|^2 d\mu \right)^{1/2}&\le&  \left(\aver{B(x_0,r)} |\nabla u|^2 d\mu \right)^{1/2}+ \left(\aver{B(x_0,r)} |\nabla (f-u)|^2 d\mu \right)^{1/2}\\
 &\le&  \left(\aver{B(x_0,r)} |\nabla u|^2 d\mu \right)^{1/2}+ \left(\frac{R}{r}\right)^{\frac{\nu}{2}}\left(\aver{B(x_0,R)} |\nabla (f-u)|^2 d\mu \right)^{1/2}\\
 &\le& \left(\frac{R}{r}\right)^{1-\alpha} \left(\aver{B(x_0,R)} |\nabla u|^2\,d\mu\right)^{1/2}
+ \left(\frac{R}{r}\right)^{\frac{\nu}{2}}\left(\aver{B(x_0,R)} |\nabla (f-u)|^2 d\mu \right)^{1/2}
.\end{eqnarray*}
Since \eqref{FK} follows from our assumptions, we can use \eqref{sarko} and together with    \eqref{fillon} it follows   that 
\begin{eqnarray*}
\left(\aver{B(x_0,r)} |\nabla f|^2 d\mu \right)^{1/2} 
& \lesssim& \left(\frac{R}{r}\right)^{1-\alpha} \left[\left(\aver{B(x_0,R)} |\nabla f|^2 d\mu\right)^{1/2}+R \left(\aver{B(x_0,R)} |\mathcal{L}f|^2d\mu \right)^{1/2}\right]\\
&&+ \left(\frac{R}{r}\right)^{\nu/2}R \left(\aver{B(x_0,R)} |\mathcal{L}f|^2d\mu \right)^{1/2},
\end{eqnarray*}
hence
\begin{equation*}
\left(\aver{B(x_0,r)} |\nabla f|^2 d\mu \right)^{1/2} 
 \lesssim \left(\frac{R}{r}\right)^{1-\alpha} \left(\aver{B(x_0,R)} |\nabla f|^2 d\mu\right)^{1/2}+\left(\frac{R}{r}\right)^{\gamma}R \left(\aver{B(x_0,R)} |\mathcal{L}f|^2d\mu \right)^{1/2},
\end{equation*}
where $\gamma=\max\{1-\alpha,\frac{\nu}{2}\}$.
Applying  Lemma \ref{lemma-end} and since
$$ \sup_{r\leq R} \left(\aver{B(x_0,r)} |\mathcal{L}f|^2d\mu \right)^{1/2} \lesssim \|\mathcal{L}f \|_{L^\infty(B_R)},$$
 one obtains ($DG_{2,\epsilon}$) for every $\epsilon\in(1-\alpha,1)$.
\end{proof}


 \section{A  self-improving property for reverse H\"older inequalities}
  \label{App:ste}

In this appendix, we shall describe a general self-improving property of reverse H\"older inequalities which is used in Proposition \ref{prop}. These results are not new and already appeared in the literature (see  e.g. \cite[Theorem 2]{IN} and \cite[Subsection 3.38]{HKM}). We give a proof for the sake of completeness.
 
 Consider $(M,d,\mu)$ a doubling metric measure space.
 Let ${\mathcal Q}$ be the collection of all balls of the ambient space $M$, and consider a functional $a: {\mathcal Q} \to [0,\infty)$. We say that $a$ is regular  if there exists a constant $c>0$ such that for every pair of balls $B,\tilde B$ with $\tilde B \subset B \subset 2 \tilde B$  
 $$  c a(\tilde B) \leq  a(B)  \leq c^{-1} a(2\tilde B).$$

 \begin{theorem} \label{thm:imp} 
Let $1<p< q\leq + \infty$.  Consider a regular functional $a$.
 Let $\omega\in L^1_{\loc}(M,\mu)$ be a non-negative function such that for every ball $B\subset M$
\begin{equation} \left( \aver{B} \omega^ q \, d\mu \right)^{1/q} \lesssim \left( \aver{2B} \omega^ p \, d\mu \right)^{1/p} + a(B). \label{eq:a} 
\end{equation}
Then, for every $\eta\in(0,1)$ and every ball $B\subset M$,
 $$ \left( \aver{B} \omega^ q \, d\mu \right)^{1/q} \lesssim \left( \aver{2B} \omega^ {\eta p} \, d\mu \right)^ {1/(\eta p)} + a(2B). $$
In other words, the right-hand side exponent of a reverse H\"older inequality always self-improves.
 \end{theorem}

\begin{proof}
Fix $\eta\in(0,1)$.
For every $\epsilon\in(0,1)$, consider
$$ K(\epsilon,\eta) := \sup_{B\in {\mathcal Q}} \frac{\left( \aver{B} \omega^p \, d\mu \right)^{1/p} }{ \left( \aver{2B} \omega^ {\eta p} \, d\mu \right)^ {1/(\eta p)}  + a(2B) +  \epsilon \left( \aver{2B} \omega^p \, d\mu \right)^{1/p}}.$$
It is  easy to  observe that $K(\epsilon,\eta)$ is finite and bounded by $\epsilon^{-1}$.
We claim that $K(\epsilon,\eta)$ is uniformly bounded, with respect to $\epsilon$.

Indeed, assume that $K(\epsilon,\eta)\geq 1$ (else there is nothing to prove) and take a ball $B\in {\mathcal Q}$. Consider $(B_i)_i$ a finite collection of balls which covers $B$ with $\ell(B_i)\simeq \ell(B)$ and $4B_i \subset 2B$.
Then
\begin{align*}
 \left( \aver{B} \omega^p \, d\mu \right)^{1/p}  \lesssim \sum_i \left( \aver{B_i} \omega^p \, d\mu \right)^{1/p} = \sum_i \left( \aver{B_i} (\omega^ \delta \omega ^{1-\delta})^p \, d\mu \right)^{1/p},
 \end{align*}
 for any $\delta\in[0,1]$.
Using H\"older's inequality with the particular choice $\delta =\frac{\eta(q-p)}{q-\eta p}$ gives us
\begin{align*}
\left( \aver{B_i} (\omega^ \delta \omega ^{1-\delta})^p \, d\mu \right)^{1/p} & \leq \left( \aver{B_i} \omega^ {\eta p} \, d\mu \right)^{\frac{\delta}{\eta p}} \left( \aver{B_i} \omega^q \, d\mu \right)^{(1-\delta)/q} ,
\end{align*}
since then
$$ \frac{1}{p} = \frac{\delta}{\eta p} + \frac{1-\delta}{q}.$$
Using \eqref{eq:a}, the assumption $K(\epsilon,\eta)\geq 1$ and the fact that $4B_i \subset 2B$, one obtains
\begin{align*}
& \left( \aver{B_i} (\omega^ \delta \omega ^{1-\delta})^p \, d\mu \right)^{1/p} 
 \lesssim \left( \aver{B_i} \omega^ {\eta p} \, d\mu \right)^{\frac{\delta}{\eta p}} \left[\left( \aver{2B_i} \omega^p \, d\mu \right)^{1/p}  + a(B_i) \right]^ {1-\delta}\\
&  \lesssim  \left( \aver{B_i} \omega^ {\eta p} \, d\mu \right)^{\frac{\delta}{\eta p}}  
K(\epsilon,\eta)^{1-\delta}\left[ \left( \aver{4B_i} \omega^ {\eta p} \, d\mu \right)^ {1/(\eta p)}  + \epsilon \left( \aver{4B_i} \omega^p \, d\mu \right)^{1/p}  + a(2B) \right]^{1-\delta}.
\end{align*}
By summing over the finite collection of balls $(B_i)$, we deduce that
\begin{align*}
&\left( \aver{B} \omega^p \, d\mu \right)^{1/p}\\ 
& \lesssim  \left(\aver{B} \omega^ {\eta p} \, d\mu \right)^{\frac{\delta}{\eta p}} 
K(\epsilon,\eta)^{1-\delta}  \left[ \left( \aver{2B} \omega^ {\eta p} \, d\mu \right)^ {1/(\eta p)}  + \epsilon \left( \aver{2B} \omega^p \, d\mu \right)^{1/p}  + a(2B) \right]^ {1-\delta} \\
& \lesssim  K(\epsilon,\eta)^{1-\delta}  \left[ \left( \aver{2B} \omega^ {\eta p} \, d\mu \right)^ {1/(\eta p)}  + \epsilon \left( \aver{2B} \omega^p \, d\mu \right)^{1/p} + a(2B)\right].
\end{align*}
Taking the supremum over all balls $B$ then yields
\begin{align*}
K(\epsilon,\eta)\lesssim K(\epsilon,\eta)^{1-\delta},
\end{align*}
which in turn yields, since $K(\epsilon,\eta)$ is finite  and $\delta>0$ due to $p<q$,
$$ K(\epsilon,\eta) \lesssim 1.$$
This last estimate is uniform with respect to $\epsilon$.
Hence, by letting $\epsilon \to 0$, we obtain the desired conclusion.
\end{proof}

 \DeactivateToc
\section*{Acknowledgements}
\ActivateToc

The first and the third author acknowledge financial support from the Hausdorff  Research Institute for Mathematics in Bonn where part of this work was done. All authors are very grateful to Steve Hofmann and Jos\'e Mar\'\i a Martell for a fruitful discussion where they have pointed out the results of Appendix \ref{App:ste}.  The second author would like to thank Estibalitz Durand-Cartagena for interesting discussions about Poincar\'e inequalities on metric spaces.


\begin{thebibliography}{AAA}

\bibitem{Aus}
P. Auscher,
\newblock Regularity theorems and heat kernel for elliptic operators, 
\newblock {\it J. London Math. Soc.},  \textbf{54} (1996), no. 2, 284--296.

\bibitem{AC2} P. Auscher and T. Coulhon, Gaussian lower bounds for random walks from elliptic regularity,
\newblock {\it Annales de l'I.H.P., proba. stat.,} \textbf{35} (1999), no. 5, 605--630.

\bibitem{AC} P. Auscher and T. Coulhon, Riesz transform 
on manifolds and Poincar\'{e} inequalities, \textit{Ann. Scuola Norm. Sup. 
Pisa}, \textbf{4} (2005), 531--555.

\bibitem{ACDH} P. Auscher, T. Coulhon, X. T. Duong, and S. Hofmann, 
 Riesz transform on manifolds and heat kernel regularity,
 \textit{Ann. Sci. Ecole Norm. Sup.}, \textbf{37} (2004), no. 4, 911--957.

\bibitem{AMcM}
P.~Auscher, A.~McIntosh, and A. Morris,
\newblock{Calder\'on reproducing formulas and applications to Hardy spaces},
\newblock http://arxiv.org/abs/1304.0168.



\bibitem{BGK} M. Barlow, A. Grigor'yan, and T. Kumagai,  
On the equivalence of parabolic Harnack inequalities and heat kernel estimates, 
{\it J. Math. Soc. Japan}, {\bf 64} (2012), no. 4, 1091--1146.

 \bibitem{BCF2} 
I. Benjamini, I. Chavel, and E. Feldman, Heat kernel lower bounds on Riemannian manifolds using the old ideas of Nash,
\newblock {\it Proc. London Math. Soc.}, (3) \textbf{72} (1996), no. 1, 215--240. 

%

\bibitem{Bou0}
S. Boutayeb,
From lower to upper estimates of heat kernels in doubling spaces, {\em Tbil. Math. J.}, \textbf{2} (2009), 61--76.

\bibitem{Boutayeb}
S. Boutayeb,
Heat kernel lower Gaussian estimates in the doubling setting without Poincar\'e inequality,
{\em Publ. Mat.}, \textbf{53} (2009), 457--479.

\bibitem{BCS} S. Boutayeb, T. Coulhon, and A. Sikora, 
\newblock{A new approach to pointwise  heat kernel upper bounds on doubling metric measure spaces},  preprint 2013, http://arxiv.org/abs/1311.0367.


\bibitem{CCO} G. Carron, T. Coulhon, and E. M. Ouhabaz,
Gaussian estimates and $L^p$-boundedness of Riesz means,
\newblock {\it J. of Evol. Equ.}, \textbf{2} (2002) 299--317.


\bibitem{Chen2}
L. Chen, 
\newblock Hardy spaces on metric measure spaces with generalized heat kernel estimates, preprint.


\bibitem{CLi} T. Coulhon, Espaces de Lipschitz et in\'egalit\'es de Poincar\'e, {\it J. Funct. Anal.}, 136 (1996),
81--113.

\bibitem{coupre}
T.~Coulhon,
\newblock Off-diagonal heat kernel lower bounds without Poincar\'e,
\newblock{\em J. London Math. Soc.}, \textbf{68} (2003), no. 3, 795--816.



\bibitem{CHQL} T. Coulhon and H. Q. Li,
Estimations inf\'erieures du noyau de la chaleur sur les vari\'et\'es coniques et transform\'ee de Riesz,
\newblock {\it Arch. Math. (Basel)}, \textbf{83} (2004), no. 3, 229--242.


 \bibitem{CS} T. ~Coulhon and A. ~Sikora,
\newblock{Gaussian heat kernel bounds via Phragm\'en-Lindel\"of theorem,}  
\newblock{\it Proc. London Math. Soc.}, \textbf{96} (2008),  507--544.
 
 \bibitem{CS2} T. Coulhon and A. Sikora, 
Riesz meets Sobolev, \textit{Colloq. Math.}, \textbf{118} (2010), 685--704. 

 \bibitem{CS3} T. Coulhon and A. Sikora, 
Geometric aspects of Nash and Gagliardo-Nirenberg inequalities, preprint. 

\bibitem{DeG}
E.  De Giorgi, 
Sulla differenziabilita de analiticita delle estremali degli integrali multipli regolari, 
\newblock {\it Mem. Accad. Sci. Torino Cl. Sci. Fis. Mat. Nat.}, \textbf{3}  (1957), no. 3, 25--43.

\bibitem{Del}  T. Delmotte, Parabolic Harnack inequality and estimates of Markov chains on graphs, {\it Rev. Mat. Iberoamericana}, \textbf{15} (1999), no. 1, 181--232.


\bibitem{DOS}
X.T. Duong, E.M. Ouhabaz, and A. Sikora,
\newblock Plancherel-type estimates and sharp spectral multipliers,
\newblock {\it J. Funct. Anal.},  \textbf{196} (2002), 443--485.

\bibitem{Du0} N. Dungey,  Heat kernel estimates and Riesz transforms on some Riemannian covering
manifolds,  {\em Math. Z.},   \textbf{247} (2004), no. 4,  765--794.

\bibitem{Du}
N. Dungey, Some remarks on gradient estimates for heat kernels,
\newblock {\it Abstr. Appl. Anal.}, (2006), Art. ID 73020.

\bibitem{DCJS}
E. Durand-Cartagena, J. Jaramillo, and N. Shanmugalingam, 
\newblock The $\infty$-Poincar\'e inequality in metric measure spaces, 
\newblock {\it Michigan Math. J.}, \textbf{61} (2012), no. 1, 63--85.

\bibitem{DCJS1}
E. Durand-Cartagena, J. Jaramillo, and N. Shanmugalingam, 
\newblock Geometric characterisations of $p$-Poincar\'e inequalities in the metric setting, 
preprint 2014, http://cvgmt.sns.it/paper/2345/.

\bibitem{FS} E. Fabes and D. Stroock, A new proof  of Moser's parabolic Harnack inequality
via the old ideas of Nash, {\em Arch. Rat. Mech. Anal.},  \textbf{96} (1986), 327-338.

\bibitem{FPW} B. Franchi, C. P\'{e}rez, and R.L. Wheeden,  Self-improving properties of John-Nirenberg and
Poincar\'{e} inequalities on space of homogeneous type,
\newblock {\it J. Funct. Anal.}, \textbf{153} (1998), no. 1, 108--146.

\bibitem{FOT} 
M. ~Fukushima, Y. ~Oshima, and M. ~Takeda, \newblock{\it Dirichlet forms and symmetric Markov processes}, de Gruyter Studies in Mathematics 19, 
 \newblock {Walter de Gruyter}, \newblock Berlin, 1994.


\bibitem{Gia0}
M. Giaquinta, {\it Multiple integrals in the calculus of variations and nonlinear elliptic systems}, Annals of Mathematics Studies, 105. Princeton University Press, Princeton, NJ, 1983.

\bibitem{Gia}
M. Giaquinta, {\it Introduction to regularity theory for nonlinear elliptic systems},
Birkh\"auser, 1993.


\bibitem{Grigo}  A.  Grigor'yan, On stochastically complete manifolds, {\it Soviet Math. Dokl.}, \textbf{34} (1987), no. 2, 310--313.

\bibitem{Gr0}  A.  Grigor'yan,  
The heat equation on non-compact Riemannian manifolds, in Russian: {\em Matem. Sbornik},
 182 (1991), no. 1, 55--87 ; english translation: {\em Math. USSR Sb.}, 72 (1992), no. 1,  47--77.

\bibitem{Gr1} 
A.~Grigor'yan, 
\newblock Gaussian upper bounds for the heat kernel on arbitrary manifolds,
\newblock {\it J. Diff. Geom.}, \textbf{45} (1997), 33--52.

\bibitem{GH} \textrm{A.~Grigor'yan and J.~Hu}, \newblock{Upper bounds of heat kernels on doubling spaces}, \newblock http://www.math.uni-bielefeld.de/~grigor/pubs.htm.

\bibitem{GSC}
P. Gyrya and L. Saloff-Coste,
{\it Neumann and Dirichlet heat kernels in inner uniform domains},
Ast{\'e}risque, \textbf{33} (2011), Soc. Math. France.

\bibitem{HaKo}
P. Hajlasz and P. Koskela, 
{\it Sobolev met Poincar\'e},  Memoirs of the A.M.S., \textbf{145} (2000), no. 688.

\bibitem{HS}
W. Hebisch and L. Saloff-Coste,
On the relation between elliptic and parabolic Harnack inequalities,
\newblock {\it Ann. Inst. Fourier}, \textbf{51} (2001), no. 5, 1437--1481.

\bibitem{HK} J. Heinonen and P. Koskela,  Quasiconformal maps in metric spaces with controlled geometry, {\it Acta Math.}, \textbf{181} (1998), 1--61.

\bibitem{HKM}
J. Heinonen, T. Kilpel\"ainen, and O. Martio
\newblock {\it Nonlinear Potential Theory of Degenerate Elliptic Equations,}
\newblock Oxford University Press, 1993.

\bibitem{HKim} S. Hofmann and S. Kim, 
\newblock The Green function estimates for strongly elliptic systems of second order,
\newblock {\it Manuscripta Mathematica}, \textbf{124} (2007), 139--172.

\bibitem{IN}
T. Iwaniec and C.A. Nolder, 
\newblock Hardy-Littlewood inequality for quasiregular mappings in certain domains in ${\mathbb R}^n$,
\newblock {\it Ann. Acad. Sci. Fenn. Ser. A I Math.} \textbf{10} (1985), 267--282.

\bibitem{Iw} T. Iwaniec, The Gehring lemma, in {\em Quasiconformal mappings and 
analysis (Ann Arbor, MI, 1995)}, 181--204, Springer, New York, 1998. 

\bibitem{KZ}
S. Keith and X. Zhong,
\newblock The Poincar\'e inequality is an open ended condition,
\newblock {\it Annals of Math.}, \textbf{167} (2008), 575--599.

\bibitem{LY} P. Li and S. T. Yau, On the parabolic kernel of the
Schr\"odinger operator, {\it Acta Math.}, 156, 153-201, 1986.

\bibitem{HQL1}  H. Q. Li, 
\newblock La transformation de Riesz sur les vari\'et\'es coniques,
\newblock {\it J. Func. Anal.}, \textbf{168} (1999), 145--238.

\bibitem{HQL2}  H. Q. Li, Estimations du noyau de la chaleur sur les vari\'et\'es coniques et ses applications, 
\newblock {\it Bull. Sci. Math.}, \textbf{124} (2000), 365--384.


\bibitem{meyers}
N. G. Meyers, Mean oscillation over cubes and H\"older continuity,
\newblock {\it Proc. Amer. Math. Soc.} \textbf{15} (1964), 717--721.

\bibitem{mosco} U. Mosco, Composite media and asymptotic Dirichlet forms, {\it J. Funct. Anal.}, 123 (1994),
368--421.


\bibitem{Moser}
 J. Moser, On Harnack theorem for elliptic differential equations,
{\it Comm. Pure Appl. Math.}, \textbf{14} (1961) 577--591.

\bibitem{saloffpoly}  L. Saloff-Coste, Analyse sur les groupes de Lie \`a
croissance polynomiale, {\it Ark. Mat.}, \textbf{28} (1990), 315--331.

\bibitem{S} L. ~Saloff-Coste, A note on Poincar\'e, Sobolev, and Harnack inequalities, {\it Int. Math. Res. Not.}, (1992), no. 2,  27--38.

\bibitem{SU} L. ~Saloff-Coste, Uniformly elliptic operators on Riemannian manifolds, {\it J. Diff. Geom.}, \textbf{36} (1992), 417--450.

\bibitem{parma} L. Saloff-Coste, Parabolic Harnack inequality
for divergence form second order differential operators,
{\em Pot. Anal.}, \textbf{4} (1995), no. 4, 429--467.

\bibitem{SA} L. ~Saloff-Coste,  {\it Aspects of Sobolev-type inequalities},
London Math. Soc. Lecture Note Series 289,
Cambridge University Press, 2002.

\bibitem{Sikora}
A. Sikora,
\newblock  Riesz transform, Gaussian bounds and the method of wave equation.
\newblock {\it Math. Z.} \textbf{247} (2004), no. 3, 643--662.

\bibitem{topics} E. M. Stein, {\it Topics in harmonic analysis related to
the
Littlewood-Paley theory}, Princeton University Press, 1970.

\bibitem{ST1} K.T. Sturm, Analysis on local Dirichlet spaces I. Recurrence, conservativeness  and $L^p$-Liouville property, {\it  J. Reine Angew. Math.}, \textbf{456} (1994), 173--196.

\bibitem{ST2} K.T. Sturm, Analysis on local Dirichlet spaces II. Upper Gaussian estimates for the fundamental solutions of parabolic equations, {\it Osaka J. Math.}, \textbf{32} (1995), no. 2, 275--312.

\bibitem{ST3} K.T. Sturm, Analysis on local Dirichlet spaces III. The parabolic Harnack inequality, {\it J. Math. Pures Appl.}, \textbf{75} (1996), 273--297.




\end{thebibliography}
\end{document}